\documentclass[a4paper,10pt]{amsart}
\usepackage[plainpages=false]{hyperref}
\usepackage{amsfonts,latexsym,rawfonts,amsmath,amssymb,amsthm, mathrsfs, lscape}
\usepackage{verbatim}
\usepackage{color}

\usepackage{array, tabularx}

\usepackage{setspace}

\newtheorem{thm}{Theorem}[section]
\newtheorem{cor}[thm]{Corollary}
\newtheorem{lem}[thm]{Lemma}
\newtheorem{clm}[thm]{Claim}
\newtheorem{prop}[thm]{Proposition}
\newtheorem{quest}[thm]{Question}
\newtheorem{conj}[thm]{Conjecture}
\newtheorem{rmk}[thm]{Remark}
\newtheorem{hyp}[thm]{Hypothesis}
\newtheorem{exa}[thm]{Example}

\theoremstyle{definition}
\newtheorem{defi}[thm]{Definition}

\numberwithin{equation}{section}\def \N {\mathbb N}

\def \M {\mathcal M}
\def \L {\mathcal L}
\def \O {\mathcal O}

\def \X {\mathcal X}

\def \F {\mathbb F}
\def \A {\mathbb A}

\def \Q {\mathbb Q}
\def \C {\mathbb C}
\def \Z {\mathbb Z}
\def \P {\mathbb P}
\def \B {\mathbb B}

\def \tY {\widetilde Y}

\def \Def {\text{Def}}

\def \Aut {\text{Aut}}
\def \Kur {\text{Kur}}
\def \Hilb {\text{Hilb}}
\def \SL {\text{SL}}

\begin{document}
\title[]
{Compact Moduli Spaces of Del Pezzo Surfaces and  K\"ahler-Einstein metrics}
\author{Yuji Odaka}
\address{Department of Mathematics, Kyoto University, Kyoto 606-8502, Japan}
\email{yodaka@math.kyoto-u.ac.jp}

\author{Cristiano Spotti}
\address{DPMMS, Centre for Mathematical Sciences, Wilberforce Road, Cambridge CB3 0WB, United Kingdom}
\email{c.spotti@dpmms.cam.ac.uk / cristiano.spotti@gmail.com}

\author{Song Sun}
\address{Department of Mathematics, Stony Brook University, Stony Brook, NY 11794-3651, USA}
\email{song.sun@stonybrook.edu }

\date{\today}
\maketitle

\begin{abstract}
We prove that the Gromov-Hausdorff compactification of the moduli space of K\"ahler-Einstein Del Pezzo surfaces in each degree agrees with certain algebro-geometric compactification. In particular, this recovers  Tian's theorem on the existence of K\"ahler-Einstein metrics on smooth Del Pezzo surfaces and classifies all the degenerations of such metrics. The proof is based on a combination of both  algebraic and differential geometric techniques. 
\end{abstract}

\tableofcontents

\section{Introduction} \label{Intro}

   For each positive integer $d$, we denote by $M_d^{GH}$ the \emph{Gromov-Hausdorff compactification} of the moduli space of degree $d$ K\"ahler-Einstein Del Pezzo surfaces, and denote by $M_d^0$ the dense subset that parametrizes those smooth surfaces.   It is well-known that for $d\geq 5$ the moduli space is just a single point, so in this paper we will always assume $d\in \{1, 2, 3,4\}$. 
By Tian-Yau \cite{TY} we know that $M_d^0$ is at least a non-empty set. By general theory, $M_d^{GH}$ is a compact Hausdorff space under the Gromov-Hausdorff topology. By \cite{An, BKN, Tian1},  points in $M_d^{GH}\setminus M_d^0$ parametrise certain K\"ahler-Einstein log Del Pezzo surfaces, and a famous theorem of Tian \cite{Tian1} says that, every smooth Del Pezzo surface admits a K\"ahler-Einstein metric 
so that it is actually parametrized in $M_d^0$. 
        
 In this paper for each $d$ we identify $M_d^{GH}$ with certain explicit algebro-geometric moduli space of log Del Pezzo surfaces. The latter is 
 a compact Moishezon analytic space $M_d$, which, roughly speaking, 
parametrizes isomorphism  classes of certain $\Q$-Gorenstein smoothable log Del Pezzo surfaces of degree $d$. 
Notice that there are a-priori several possibilities of such algebro-geometric compactifications of the moduli varieties.
On the other hand, the Gromov-Hausdorff compactification is clearly canonical but very non-algebraic and just topological in nature. We refer to \cite{BBI} as an introductory textbook for those who are not familiar with 
the Gromov-Hausdorff topology. The following main theorem of the present article builds a bridge between the two notions of moduli spaces. 

\begin{thm}\label{MT}
  For each integer $d$, there is a compact moduli algebraic space \footnote{For $d\neq 1$, it follows from the construction that $M_d$ is actually a \emph{projective variety}.} $M_d$, that will be constructed explicitly in later sections, and a homeomorphism $$\Phi\colon M_d^{GH} \rightarrow M_d,$$ such that $[X]$ and $\Phi([X])$ parametrize isomorphic log Del Pezzo surfaces for any $[X]\in M_d^{GH}$. 
  Moreover, $M_d$ contains a (Zariski) open dense subset which parametrizes all smooth degree $d$ Del Pezzo surfaces. 
\end{thm}

\noindent
For the precise formulation, see Section \ref{Moduli space}. 
Theorem \ref{MT} immediately implies the above mentioned theorem of Tian, and also classifies all degenerations of K\"ahler-Einstein Del Pezzo surfaces which was posed as a problem in \cite{Tian2}. When $d=4$ Theorem \ref{MT} was proved by Mabuchi-Mukai \cite{MM}, and we shall provide a slightly different proof based on our uniform strategy. For other degrees, there have been partial  results by \cite{Chel}, \cite{CW},  \cite{GK}, \cite{Shi}, \cite{Wang} on the existence of 
K\"ahler-Einstein metrics on some canonical Del Pezzo surfaces, by calculating $\alpha$-invariant. 

A minor point is that the Gromov-Hausdorff topology defined here is slightly different from the standard definition, in that we also remember the complex structure when we talk about convergence. See \cite{DS}, \cite[Chapter 1 and 4]{Spotti} and Section \ref{DG input} for a related discussion on this.  The standard Gromov-Hausdorff compactification is homeomorphic to the quotient of $M_d$ by the involution which conjugates the complex structures. 

For the proof of Theorem \ref{MT}, we do not need to assume the existence of K\"ahler-Einstein metrics on all the smooth Del Pezzo surfaces. 
The only assumption which we need, and which has been originally proved by Tian-Yau \cite{TY}, is the following:

\begin{hyp}\label{Tian-Yau}
For each $d\in\{1,2,3,4\}$, $M_d^0$ is non-empty as a set.
\end{hyp}

Given this, the main strategy of proving Theorem \ref{MT} is as follows: 

\begin{enumerate}
\item For each $d$, we construct a natural moduli variety $M_d$ with a Zariski open subset $M_d^{\rm sm}$ parametrizing all smooth degree $d$ Del Pezzo surfaces. Moreover, there is a well-defined continuous  map $\Phi\colon M_d^{GH}\rightarrow M_d$, where we use the Gromov-Hausdorff distance in the domain and the local analytic topology in the target, so that $[X]$ and $\Phi([X])$ parametrize isomorphic log Del Pezzo surfaces for any $[X]\in M_d^{GH}$.
\item $\Phi$ is injective. This follows from the uniqueness theorem of Bando-Mabuchi \cite{BM} and its  extension to orbifolds. 
\item $\Phi$ is surjective. This  follows from the fact that the image of $\Phi$ is  open in $M_d^{\rm sm}$ (by the implicit function theorem, see for example \cite{LS}) and closed in $M_d$ (by the continuity of $\Phi$ in (1)).
\item Since $M_d^{GH}$ is compact and $M_d$ is Hausdorff,   then $\Phi$ is a homeomorphism. 
\end{enumerate}

The main technical part lies in Step (1).  For this we need  first to investigate Gromov-Hausdorff limits of K\"ahler-Einstein Del Pezzo surfaces, and then construct a  moduli space that includes all the possible limits. The difficulty increases as the degree goes down.  When $d=3, 4$ we take the classical GIT moduli space on the anti-canonical embedding.  For $d=2$ we take the  moduli space constructed in \cite{Mukai} (based on Shah's idea \cite{Shah} which blows up a certain GIT quotient). For $d=1$ we need to combine Shah's method with further modifications suggested by the differential geometric study of Gromov-Hausdorff limits. As far as we are aware,  this moduli space is new. We should mention that in the last two cases, $M_d$ (and thus $M_d^{GH}$) contains points that parametrize non-canonical log Del Pezzo surfaces. This disproves a conjecture of Tian in \cite{Tian1}, see Remark \ref{Tian conjecture}. 
We also remark that   Gromov-Hausdorff limits of K\"ahler-Einstein Del Pezzo surfaces was first 
studied by Tian in \cite{Tian1}, but as we shall see there are some inaccuracies in \cite{Tian1}, see Remark \ref{remark BG}  and Example \ref{degree1 toric}. 
 
 Finally we remark that for each $d\in \{1,2,3,4\}$, it is easy to find explicit examples of singular degree $d$ $\Q$-Gorenstein smoothable K\"ahler-Einstein log Del Pezzo surface by a global quotient construction (see the examples in later sections). Thus one way to avoid assuming Hypothesis \ref{Tian-Yau} would be to find a smooth K\"ahler-Einstein Del Pezzo surface by a gluing construction. For example, it has been proved in \cite{Spotti2} that for a K\"ahler-Einstein log Del Pezzo surface with only nodal singularities and discrete automorphism group, one can glue model Eguchi-Hanson metrics to obtain nearby K\"ahler-Einstein metrics in the smoothing. This can be applied when $d=3$,  since the Cayley cubic (see Section \ref{Degree34}) satisfies these assumptions. 
 
The organization of this paper is as follows. In Section \ref{DG input} we collect the main results that we need on the structure of Gromov-Hausdorff limits, focusing on the two dimensional case. In Section \ref{AG input} we make an algebro-geometric study of the Gromov-Hausdorff limits, and define precisely the notion of moduli spaces that we use in this paper. Then we reduce the proof of Theorem \ref{MT} to the construction of moduli spaces in each degree.  In later Sections we treat the cases $d\geq 3$ and $d\leq 2$ separately. We also investigate the relation with moduli space of curves, in 
subsections \ref{curve.2}, \ref{curve.1}. In Section \ref{K moduli} we make some further discussions. \\

\textbf{Notation}: 

A \emph{Del Pezzo surface} is a smooth projective surface with ample anti-canonical bundle. A \emph{log Del Pezzo surface} is a normal projective surface with quotient singularities (or equivalently, with log terminal singularities) and ample anti-canonical divisor.  For a log Del Pezzo surface $X$, its degree $\deg(X)$ is the intersection number $K_X^2$. In  general dimensions, a {\textit{$\mathbb{Q}$-Fano variety}} means a normal projective variety with log terminal singularities and with $-rK_X$ ample for some positive integer $r$. Smallest such $r$ will be called {\textit{index}} or 
{\textit{Gorenstein index}}. \\

\textbf{Acknowledgements}: This work is motivated by the PhD Thesis of the second named author under the supervision of Professor Simon Donaldson. We would like to  thank him for great support. The pre-print version of this paper was written when all the authors were based at Imperial College London. We would also like to thank Professors Jarod Alper, Claudio Arezzo, 
Paolo Cascini, Ivan Cheltsov, Xiuxiong Chen, Mark Haskins,
David Hyeon, Alexander Kasprzyk, Radu Laza, 
Yongnam Lee, Shigeru Mukai, Hisanori Ohashi, Shingo Taki and Bing Wang for  helpful discussions and encouragements.  S.S. was partly  funded by European Research Council award No 247331.

\section{General results on the Gromov-Hausdorff limits}\label{DG input}

The main differential geometric ingredient involved in the proof of the main theorem is the study of the structure of Gromov-Hausdorff limits of K\"ahler-Einstein Del Pezzo surfaces. 
The following orbifold compactness theorem is well-known.

\begin{prop}[\cite{An}, \cite{BKN}, \cite{Tian1}]\label{orbifold compactness}
Given a sequence of degree $d$ K\"ahler-Einstein Del Pezzo surfaces $(X_i, \omega_i, J_i)$ then, by passing to a subsequence, it converges in the Gromov-Hausdorff sense to a K\"ahler-Einstein log Del Pezzo surface $(X_\infty, \omega_\infty, J_\infty)$, and $\deg(X_\infty)=d$.
\end{prop}

In \cite{Tian1} Tian found further constraints on  the possible singularities that could appear in $X_\infty$. We will state a more general theorem and give an alternative proof. 
First we have (compare also \cite{Tian4}):

\begin{prop}[\cite{DS}]\label{Fano limit}
Given a sequence of $n$-dimensional K\"ahler-Einstein Fano manifolds $(X_i, \omega_i, J_i)$, by passing to a subsequence,  it converges in the Gromov-Hausdorff sense to a $\Q$-Fano variety $(X_\infty, J_\infty)$ endowed with a weak K\"ahler-Einstein metric $\omega_\infty$ (cf. \cite{EGZ}). Moreover, there exist integers $k$ and $N$, depending only on $n$, so that we could embed $X_i\,(i\in  \N \cup \{\infty\})$ into  $\P^N$ using orthonormal basis of $H^0(X_i, -kK_{X_i})$ with respect to the Hermitian metric  defined by $\omega_i$, and $X_i$ converges to $X_\infty$ as varieties in $\P^N$. 
 \end{prop}

Here one can think of the convergence as varieties in $\P^N$ as the convergence of defining polynomials. Notice that the orbifold property in Proposition \ref{orbifold compactness} also follows naturally from Proposition \ref{Fano limit}, since by Kawamata's theorem \cite{Kawa} that a two dimensional log terminal singularity is a quotient singularity. \\

We will treat singular varieties that come from certain limits of smooth ones. The following algebro-geometric notion is very natural from the point of view of minimal model program, and will be shown to be also naturally satisfied by the above limit $X_\infty$. 

\begin{defi}
Let $X$ be a $\Q$-Fano variety. We say $X$ is \emph{$\Q$-Gorenstein smoothable} if  there exists a deformation $\pi: \X\rightarrow \Delta\ni 0$ of $X$ over a smooth curve germ $\Delta$  such that $\X_0=X$, the general fibre is smooth and $K_{\X}$ is $\Q$-Cartier. 
\end{defi}

\begin{lem} \label{Qsmoothable}
 $X_\infty$  is $\Q$-Gorenstein smoothable. 
\end{lem}

\begin{proof} By Proposition \ref{Fano limit} and general theory we can find a family of varieties  $\pi_2: \X\subset \P^N\times \Delta \rightarrow \Delta$ in $\P^N$ where for $t\neq 0$  $\X_t$ is smooth and $\X_0$ is the variety $X_\infty$. 
Indeed, for a morphism from $\Delta$ to the Hilbert scheme which sends 
$0$ to $X_\infty$ (embedded by $|-kK_{X_\infty}|$) and contains 
$X_i$ (embedded by $|-kK_{X_i}|$ as well) for one sufficiently large $i$, we can construct the required family by pulling back 
the total space and take its normalization if necessary.
Denote the other projection map by $\pi_1\colon \P^N\times \Delta\rightarrow \P^N$, then $-rK_{\X}$ and $\pi_1^*\O(1)$ agrees up to a pull back from the base. Thus $-rK_{\X}$ is Cartier and so $X_\infty$ is $\Q$-Gorenstein smoothable. 
\end{proof}

Note that the above proof  does not use the $\mathbb{Q}$-Gorenstein property of the normal central fiber, 
although in our case, we knew it by Proposition \ref{Fano limit}. We only need a relatively ample linear line bundle and the normality assumption 
of the total space and the central fiber.  
The definition of $\Q$-Gorenstein smoothability can be obviously defined also for local singularities, and for a $\Q$-Gorenstein smoothable $\Q$-Fano manifold, all its singularities must also be $\Q$-Gorenstein smoothable.  On the other hand, it is proved in \cite{HP} that a log Del Pezzo surface with $\Q$-Gorenstein smoothable singularities is $\Q$-Gorenstein smoothable. In dimension two,  $\Q$-Gorenstein smoothable quotient singularities are also commonly called \emph{``T-singularities"}. The classification  of $T$-singularities is well-known, see \cite{KS}, \cite{Ma} for example. 
 So combining the above discussions we obtain:
 
\begin{thm}[\cite{Tian1}] \label{T-singularity}
 The Gromov-Hausdorff limit $(X_\infty, J_\infty)$ of  a sequence of K\"ahler-Einstein Del Pezzo surfaces is a K\"ahler-Einstein log Del Pezzo surface with singularities either  canonical (i.e. ADE singularities) or  cyclic quotients of type $\frac{1}{dn^2}(1,dna-1)$ with $(a,n)=1$ ($1\leq a < n$). 
\end{thm}

\begin{rmk}
For sake of completeness, even if we will not use it in our proof, we should remark that it is known  that  local smoothings of $T$-singularities admit asymptotically conical Calabi-Yau metrics \cite{Kr} \cite{Suvaina}. It is then natural to expect from a metric perspective the following picture: given a sequence $(X_i,\omega_i)$ of degree $d$ K\"ahler-Einstein Del Pezzo surfaces  Gromov-Hausdorff converging  to a singular $(X_\infty,\omega_\infty)$ and choose $p_\infty \in \mbox{Sing}(X_\infty)$, then there exists a sequence of points $p_i \in X_i \rightarrow p_\infty  \in  X_\infty $ and scaling parameters $\lambda_i \rightarrow +\infty$ such that $(X_i,p_i, \lambda_i \omega_i)$ converges in the pointed Gromov-Hausdorff sense to an asymptotically conical Calabi-Yau metric on a smoothing of the $T$-singularity at $p_\infty$. 
\end{rmk}

Next, we can use the Bishop-Gromov volume comparison Theorem to control the order of the orbifold group at each point. 

\begin{thm} [\cite{Tian1}]
\label{Bishop-Gromov} Let $(X, \omega)$ be a K\"ahler-Einstein log Del Pezzo surface and let $\Gamma_p \subseteq U(2)$ be the orbifold group at a point $p\in X$. Then
\begin{equation} \label{order bound}
|\Gamma_p| \deg(X) < 12,
\end{equation}
\end{thm}
\begin{proof}
Without loss of generality we may normalize the metric so that $Ric (\omega)= 3\omega$. The Bishop-Gromov volume comparison extends without difficulty to  orbifolds \cite{Bor}, so for all $p \in X$ the function
$\frac{Vol(B(p,r))}{Vol(\overline{B}(r))}$
is decreasing in $r$,  where $\overline{B}(r)$ is the ball of radius $r$ in the standard four sphere $S^4(1)$. As $r$ tends to zero the function converges to $1/|\Gamma_p|$, and for sufficiently large $r$ the function is constant $Vol(X, \omega)/Vol(S^4(1))$. 
So $Vol(X, \omega) |\Gamma_p| \leq Vol (S^4(1)).$
 The normalization condition $Ric (g)= 3 g$ implies that $[\omega] = \frac{2 \pi}{3} c_1(X)$.
 So
$Vol(X, \omega)=\int_X \frac{\omega^2}{2}=\frac{2 \pi^2}{9} \deg(X)$. 
Then, using the fact that $Vol(S^4(1))=\frac{8}{3} \pi^2$,  it is easy to see 
$|\Gamma_p| \deg(X) \leq 12.$ 
If the equality is achieved, then $X$ must have constant curvature. 
But since $X$ is K\"ahler, we have $S(\omega)^2=24|W^+|^2$, so the scalar curvature vanishes. 
Contradiction. 
\end{proof}

\begin{rmk} \label{remark BG}
The two theorems above were essentially known to Tian \cite{Tian1}. For the inequality \ref{order bound}, the constant on the right hand side was $48$ in \cite{Tian1}. 
\end{rmk}

By Theorem \ref{T-singularity} and \ref{Bishop-Gromov}, and we have the constraints on the possible singularities that could appear on the Gromov-Hausdorff limit $X_\infty$\footnote{Recall that the order of the finite Klein group yielding an $A_k$ singularity is $k+1$, a $D_k$ singularity is $4(k-2)$, an $E_6$ singularity is $24$, an $E_7$ singularity is  $48$, and an $E_8$ singularity is  $120$.}

\begin{itemize}
\item $\deg=4$, $X_\infty$ is canonical, and can have only $A_1$ singularities.
\item $\deg=3$, $X_ \infty$ is canonical, and can have only $A_1$ or $A_2$ singularities. 
\item $\deg=2$, $X_\infty$ can have only $A_1$, $A_2$, $A_3$, $A_4$, and $\frac{1}{4}(1,1)$ singularities. 
\item $\deg=1$, $X_\infty$ can have only $\frac{1}{4}(1,1)$, $\frac{1}{8}(1,3)$, and $\frac{1}{9}(1,2)$ singularities besides $A_i\ (i\leq 10)$ and $D_4$ singularities.
\end{itemize}
For the case when $d\geq 3$, the above classification is already sufficient for our purposes, as  canonical log Del Pezzo surfaces are  classified (see the next section). When $d\leq 2$ we will make a further study in Section  \ref{Degree12}. 
Now we make a side remark about the Gromov-Hausdorff topology used in this paper. In \cite{Spotti} it is proved that if two K\"ahler-Einstein log Del Pezzo surfaces are isometric, then the complex structures could be the same or conjugate. For this reason the standard Gromov-Hausdorff distance can not distinguish two conjugate complex structures in general. So in our case there is an easy modification, where we say a sequence $(X_i, J_i, \omega_i)$ converges to $(X_\infty, J_\infty, \omega_\infty)$ if  it converges in the Gromov-Hausdorff topology and in the sense of Anderson-Tian, i.e. smooth convergence of both the metric and complex structure away from the singularities. The spaces $M_d$ appearing in Theorem \ref{MT} admit an involution given by conjugating the complex structure, and we will identify explicitly this involution for each $d$.

\section{Algebro-geometric properties of log Del Pezzo surfaces}\label{AG input}

 We continue to study the algebro-geometric properties of $X_\infty$ appearing in the last section. These constraints help the construction of  the desired moduli spaces in later sections. 

\subsection{Classification of mildly singular log Del Pezzo surfaces}
We first recall some general classification results for log Del Pezzo surfaces with mild singularities.  The following is classical.

\begin{thm}[\cite{HW}] \label{logDelPezzo-canonical}
A degree $d$ log Del Pezzo surface with canonical singularities is 
\begin{itemize}
\item a complete intersection of two quadrics in $\P^4$, if $d=4$;
\item  a cubic hypersurface in $\P^3$, if $d=3$;
\item a degree $4$ hypersurface in $\P(1,1,1,2)$ not passing  $[0:0:0:1]$, if $d=2$;
\item a degree $6$ hypersurface in $\P(1,1,2,3)$ not passing $[0:0:1:0]$ and $[0:0:0:1]$, if $d=1$.
\end{itemize}
\end{thm}

Although we will not use it, log Del Pezzo surfaces with Gorenstein index two are also classified, by  \cite{AN} and \cite{Nakayama}. In case the degree is one or two, we have:

\begin{thm}[\cite{KK}] \label{logDelPezzo-index2}
A degree $2$ log Del Pezzo surface with Gorenstein index at most two is either a degree $4$ hypersurface in $\P(1,1,1,2)$, or a degree $8$ hypersurface in $\P(1,1,4,4)$. 
A degree $1$ log Del Pezzo surface with Gorenstein index at most two is a degree $6$ hypersurface in $\P(1,1,2,3)$. 
\end{thm}

Notice that by the restrictions on Gromov-Hausdorff limits of K\"ahler-Einstein Del Pezzo surfaces discussed in the previous section, we know that the Gorenstein index of such limits is less then or equal to $2$ for degree $\geq$ 2, and at most $6$ in the degree $1$ case.

\subsection{CM line bundle comparison}

In this subsection we study  GIT stability of  K\"ahler-Einstein log Del Pezzo surfaces. For smooth K\"ahler-Einstein manifolds, it is known that they are K-polystable (cf. \cite{Tian3, Stoppa, Mab}). This has been generalized to the singular setting in \cite{Berman}, and we state the two dimensional case here:

\begin{thm}[\cite{Berman}]
\label{KEtoKstability} A log Del Pezzo surface admitting a K\"ahler-Einstein metric is K-polystable. 
\end{thm}
  
  Next we state a general theorem relating K-polystability and usual GIT stabilities, using the CM line bundle of Paul-Tian \cite{PT}. 
Recall that the CM line bundle is a line bundle defined on base scheme of each flat family of polarized varieties in terms of the Deligne pairing  and if the family is $G$-equivariant with an algebraic group $G$, the line bundle naturally inherits the group action. 
It gives a GIT weight interpretation to the Donaldson-Futaki invariant whose positivity is roughly the K-stability. 
A point is that the CM line bundle is \textit{not} even nef in general so that we cannot apply GIT straightforward. 
We refere to \cite{PT}, \cite{PRS} for more details. 
  
\begin{thm}\label{CM stability} Let $G$ be a reductive algebraic group without nontrivial characters.  Let $\pi\colon (\mathcal{X},\mathcal{L})\rightarrow S$ be a $G$-equivariant 
polarized projective flat family of equidimensional varieties over a projective 
scheme.
Here  ``polarized" means that $\mathcal{L}$ is a 
relatively ample line bundle on $\mathcal{X}$, and ``equidimensional" means that all the 
irreducible components have the same dimension. 
Suppose that 
\begin{enumerate}
\item the Picard rank $\rho(S)$ is one; 
\item there is at least one K-polystable $(\mathcal{X}_t,\mathcal{L}_t)$ 
which degenerates in $S$ via a one parameter subgroup $\lambda$ in $G$, i.e. the corresponding test configuration is not 
of product type. 
\end{enumerate}
Then a point $s\in S$ is GIT (poly, semi)stable if $(\X_s,\mathcal{L}_s)$ 
is  K-(poly, semi)stable. 
\end{thm}

\begin{proof}
Let $\Lambda_{CM}$ be the CM line bundle \cite{PT} over $S$  associated to $\pi$. 
In general, this is a $G$-linearized $\mathbb{Q}$-line bundle. Let $\Lambda_0$ be the positive generator of $Pic(S)$, then  there exists integers $r>0$ and $k$, so that 
$\Lambda_{CM}^{\otimes r}\cong \Lambda_0^{\otimes k}$. The isomorphism is $G$-equivariant by the condition that $G$ has no nontrivial character. 
On the other hand, from the condition (2), we know that the degree of CM line along the  closure of the $\lambda$-orbit is positive.  This is because  by \cite{Wan} 
the degree is the sum of the Donaldson-Futaki invariant on the two degenerations along $\lambda$ and $\lambda^{-1}$. 
This implies that the integer $k$ is positive. 
Therefore,  $\Lambda_{CM}^{\otimes r}$ is ample. 

If $\pi\colon \mathcal{X}\rightarrow S$ is
the universal polarized family over a Hilbert scheme, and $G$ is the associated special linear group $SL$, 
then it is known \cite{PT} that for any $s\in S$ and one parameter subgroup
$\lambda\colon \mathbb{C}^{*}\rightarrow G$, the associated Donaldson-Futaki
invariant \cite{Do1} $DF((\mathcal{X}_s, \mathcal{L}_s); \lambda)$ is
the GIT weight in the usual sense with respect to the CM line bundle $\Lambda_{CM}^{\otimes r}$, up to a positive multiple. 
This fact can be extended to our general family $\pi\colon (\mathcal{X},\mathcal{L})
\rightarrow S$ in a straightforward way by considering $G$-equivariant morphism into a certain Hilbert scheme 
defined by $(\mathcal{X},(\pi^{*}\Lambda_0)^{\otimes l}\otimes
\mathcal{L}^{\otimes m})$
for $l\gg m\gg 0$. 
If $\mathcal{X}_s$ is reduced, from our equidimensionality assumption 
on all fibers, we can not get \textit{almost trivial} test configurations from one parameter subgroup of $G$ (in the sense of \cite{LX}, \cite{Od3}). 
This is because the central fiber of an almost trivial test configuration for a reduced 
equidimensional variety should have an embedded component. 
Summarizing up, the conclusion follows from the Hilbert-Mumford numerical criterion.
\end{proof}

We believe Theorem \ref{CM stability} should have more applications in the explicit study of general cscK metrics beyond our 
study of log Del Pezzo surfaces in this paper. For instance, there are many examples of 
equivariant family of polarized varieties parametrized by a projective space or Grassmanian through various covering constructions. In these situations one can always apply Theorem \ref{CM stability}. 
We  remark that in the above proof what we really need is the CM line bundle to be ample. For example,  the following has been known to Paul-Tian long time ago:
\begin{cor}[{\cite{Tian2.5}}] \label{local versal}
A  hypersurface $X\subseteq \mathbb{P}^N$
is  Chow polystable (resp. Chow semistable) if  $(X,\mathcal{O}_X(1))$ is K-polystable (resp. K-semistable). 
\end{cor}

\noindent

Hence, in particular, combined with \cite[Theorem 1.2]{Od2}, it follows that 
semi-log-canonical hypersurfaces with ample canonical classes and 
log-terminal Calabi-Yau hypersurfaces are GIT stable. 
This is just one of the easiest examples of applications of Theorem \ref{CM stability}.

We also state the following \textit{local} version of Theorem \ref{CM stability}, 
which we also believe to be a fundamental tool for future developments. 

\begin{lem} \label{local CM stability}
Let $S$ be an affine scheme, and $G$ be a reductive algebraic group acting on $S$
 fixing $0\in S$. Let  $\pi\colon (\X, \L)\twoheadrightarrow S$ be a  $G$-equivariant polarized flat projective 
deformation of a K-polystable reduced polarized variety $(\mathcal{X}_0,\mathcal{L}_0)$ 
and suppose all fibers $\X_s$ are  equidimensional varieties. 
We assume that if $(\mathcal{X}_{s_{1}},\mathcal{L}_{s_{1}})$
is isomorphic to $(\mathcal{X}_{s_{2}},\mathcal{L}_{s_{2}})$ then 
$s_{2}\in Gs_{1}$ and that $G$-action on $S$ is faithful. 
Then there is an affine neighborhood $S'$ of $0$ so that
the CM line bundle is equivariantly trivial over $S'$.
For such $S'$, it holds that
a point $s\in S'$ is GIT (poly)stable if  $(\X_s,\mathcal{L}_s)$ is 
K-(poly)stable. \end{lem}

This follows from similar arguments as in the proof of Theorem \ref{CM stability},  and we write a detailed proof here for the convenience of readers. 

\begin{proof}
First, let us prove that the CM line bundle $\lambda_{{\it CM}}$  is locally $G$-equivariant trivial 
around $0\in S$. 
Pick an arbitrary but sufficiently ample $G$-equivariant line bundle $\lambda_{B}$ and put  
$\lambda_{A}:=\lambda_{{\it CM}}\otimes \lambda_{B}$. We can assume $\lambda_{B}$ is also ample. 
By multiplying an appropriate character of $G$ 
to change the $G$-linearisations on $\lambda_{A}$ and $\lambda_{B}$ 
if necessary, we can assume the $G$-action at $\lambda_{A}|_{0}$ and $\lambda_{B}|_{0}$ is trivial. 
It is possible since from our assumption, $G$-action on $\lambda_{{\it CM}}|_{0}$ 
is trivial. 
Embed $S$ via $\lambda_{A}$ into projective space and take a compactification 
simply as 
the Zariski closure $\bar{S}_A$. Then applying the Hilbert-Mumford criterion to $(\bar{S}_A,\lambda_{A})$, 
we can see that there is some $G$-invariant section $s_1$ which does not vanish at $0$. It is indeed the original definition of GIT semi-stability. Note this implies the equivariant triviality of $\lambda_{A}$ at the locus where $s_1$ does not vanish. Do the same for $\lambda_{B}$ and we get $s_2$, a $G$-invariant section of $B$ non-vanishing around $0$. 
Then we let $S'$ be the locus where both $s_1$ and $s_2$ do not vanish. This is of course affine again. Over $S'$, $A$ and $B$ are both equivariantly trivial so is the difference which is exactly our CM line bundle. 

Suppose $s\in S'$ is not polystable in the GIT sense, although 
$(\X_s,\mathcal{L}_s)$ is K-polystable. The non-polystability implies that 
it degenerates to $s'\in (S'\setminus Gs)$ via one parameter subgroup 
$\lambda\colon \mathbb{C}^{*}\rightarrow G$. This gives a non product 
test configuration but its Donaldson-Futaki invariant vanishes, 
due to the equivariant triviality of the CM line bundle over $S'$. 
This contradicts. So we proved the assertion. 
\end{proof}

One can often apply this to the versal deformation family, as we do in the Section \ref{Versal deformations}. We remark that in general if the base $S$ is 
replaced by an open analytic subset of $S$, 
we can also work locally analytically provided $S$ is smooth at $0$. Let $K$ be a maximal compact subgroup of $G$, and $A$ be the tangent space of $S$ at $0$. Since $G$ fixes $0$, it induces a linear $G$ action on $A$. We fix a $K$-invariant Hermitian metric on $A$.  Then by standard slice theory for compact group actions we can find a $K$-invariant analytic neighborhood $U$ of $0$ in $S$,  a ball $B_r(0)$ in $A$, and a $K$-equivariant bi-holomorphic map from $U$ to $B_r(0)$. So we can identify $U$ with $B_r(0)$, in particular, the CM line bundle $\lambda_{CM}$ restricts to $B_r(0)$. Our statement then becomes that a point $s\in B_r(0)$ is GIT (poly)stable if $(\X_s, \mathcal{L}_s)$ is K-(poly)stable. To prove this, by making $r$ small we may choose a non-vanishing holomorphic section $s$ of $\lambda_{CM}$ over $B_r(0)$. Now define $\tilde s(x)=\int_{K} g^{-1}. s(g.x) dg$, where $dg$ is a bi-invariant Harr measure on $K$.
Then $\tilde s$ is a $K$-invariant holomorphic section over $B_r(0)$, and hence $G$-
invariant (more precisely, it is invariant in the Lie algebra level since ${\it Lie}(G)={\it Lie}(K)^{\C}$). On the other hand, since $G$ fixes zero, and $\X_0$ is K-polystable, $G$ acts trivially on the fiber of $\lambda_{CM}$ over 0. So $\tilde s(0)=s(0)\neq 0$. By making $r$ smaller again we may assume $\tilde s$ is also nowhere vanishing over $B_r(0)$. Now suppose $x\in B_r(0)$ is not polystable in the GIT sense, then there is a unique polystable orbit $G. x'$ in the closure of the $G$ orbit of $x$.
Moreover by the Kempf-Ness theorem it is easy to see that we may assume $x'$ is also in $B_r(0)$ (for example, we can choose $x'$ to satisfy the moment map equation $\mu(x')=0$), and there is a one parameter subgroup $\lambda \colon \mathbb{C}^{*}\rightarrow G$ that degenerates $x$ to $x'$. Since $\tilde{s}$ is non-vanishing and $Lie(G)$-invariant, it follows that the Donaldson-Futaki invariant as the weight of the action of the $\C^{*}$ action on the fiber of $\lambda_{CM}$ over $x'$ must vanish. By 
assumption $\X_{x'}$ is not isomorphic to $\X_s$, this implies $\X_x$ is not K-polystable.  

\subsection{Semi-universal $\Q$-Gorenstein deformations} \label{Versal deformations}
In this subsection we provide some general theory on $\Q$-Gorenstein deformations, continuing Section \ref{DG input}. 
Most of the general theory we review in the former half of this subsection 
should be well-known to experts of deformation theory but we review them for the convenience. 
First, the following is well-known, see for example the Main Theorem in page $2$ of \cite{Ma}: 

\begin{lem}[\cite{KS}, \cite{Ma}]\label{local no obstruction}
A $T$-singularity has a smooth semi-universal $\Q$-Gorenstein deformation. 
\end{lem}

A $T$-singularity is either Du Val (ADE type), which is a hypersurface singularity in $\C^3$ and has a smooth semi-universal $\Q$-Gorenstein deformation space, or a cyclic quotient of type $\frac{1}{dn^2}(1, dna-1)$ with $(a, n)=1$. The latter is the quotient of the Du Val singularity $A_{dn-1}$ by the group $\Z/n\Z$. More precisely, an $A_{dn-1}$ singularity embeds as a hypersurface  $z_1z_2=z_3^{dn}$ in $\C^3$. The generator $\zeta_n$ of $\Z/n\Z$ acts on $\C^3$ by $\zeta_n. (z_1, z_2, z_3)=(\zeta_n z_1, \zeta_n^{-1} z_2,  \zeta_n^a z_3)$, where $\zeta_n$ is the $n$-th root  of unity. One can explicitly write down a semi-universal $\mathbb{Q}$-Gorenstein deformation as the family of hypersurfaces in $\C^3$ given by $z_1z_2=z_3^{dn}+a_{d-1} z_3^{(d-1)n}+\cdots+ a_0$, see \cite{Ma}. Then its dimension is $d$. 

Moreover, it is also known that T-singularities are the only quotient surface singularities which admit $\mathbb{Q}$-Gorenstein smoothings.

Furthermore, for log Del Pezzo surface $X$ which only has T-singularities, 
since $H^{2}(T_{X})=0$ (cf., e.g., \cite[Proposition 3.1]{HP}), 
we know that $X$ has global $\mathbb{Q}$-Gorenstein smoothing. 
Summarising up, we have

\begin{lem}[\cite{KS}, Proposition 3.10 and \cite{HP}, Proposition 2.2]
Let $X$ be a log Del Pezzo surface. Then $X$ is $\Q$-Gorenstein smoothable if and only if it has only $T$-singularities. 
\end{lem}

We give more precise structure of the deformations as follows, 
which should be certainly known to experts. 

\begin{lem}\label{local-global} Let $X$ be a $\Q$-Gorenstein smoothable log Del Pezzo surface with singularities $p_1, \cdots, p_n$. Then for the $\mathbb{Q}$-Gorenstein 
deformation tangent space $\Def(X)$ of $X$, we have 
$$0\rightarrow \Def '(X)\rightarrow \Def(X) \rightarrow \bigoplus_{i=1}^n \Def_i\rightarrow 0,  $$
where $\Def'(X)$ is the subspace of $\Def(X)$ corresponding to equisingular deformations, and  $\Def_i$ is the $\Q$-Gorenstein deformation tangent space of the local singularity $p_i$. Notice $Aut(X)$ naturally acts on $\Def'(X)$ as well 
as on $\Def(X)$. 

Moreover, if $\Aut(X)$ is reductive group, there is an 
affine algebraic scheme $(\Kur(X), 0)$ with tangent space $\Def(X)$ at $0$, and a semi-universal $\mathbb{Q}$-Gorenstein family $\mathcal{U}\rightarrow (\Kur(X), 0)$ which is $\Aut(X)$-equivariant, and the induced action on $\Def(X)$ is the natural one as above. 
Here $\Aut(X)$ denotes the automorphism group of $X$. 
\end{lem}

We give a sketch of the proof here. 
In general there is a tangent-obstruction theory for deformation of singular reduced 
varieties, with tangent space 
$\text{Ext}^1(\Omega_X , \O_X )$ and obstruction space $\text{Ext}^2(\Omega_X , \O_X)$. Since $X$ has only isolated singularities and $H^2(X, T_{X})=0$ 
(cf., \cite[Proposition 3.1]{HP}) 
in which the local-to-global obstructions lie in general, 
we have the following natural exact sequnce due to the local-to-global spectral sequence of $\text{Ext}$: 
$$
0\rightarrow H^{1}(T_{X})\hookrightarrow 
{\it Ext}^{1}(\Omega_{X},\mathcal{O}_{X})\twoheadrightarrow 
\oplus_{x\in X}\mathcal{E} xt^{1}(\Omega_{X},\mathcal{O}_{X})\rightarrow 0. 
$$
It is well-known that $H^{1}(T_{X})={\it Def}'(X)$ i.e., 
it is the first order deformation tangent space 
of equisingular deformations. 
The local obstruction for deforming singularities lies in the map 
$$H^0(\mathcal{E} xt^1(\Omega_X, \O_X))=\bigoplus_{i=1}^n\mathcal{E} xt^1_{p_i}(\Omega_X, \O_X) \rightarrow H^0(\mathcal{E}xt^2(\Omega_X, \O_X))=\bigoplus_{i=1}^n\mathcal{E} xt^2_{p_i}(\Omega_X, \O_X).$$
Hence, restricting the above exact sequence to 
the subspace of $\oplus_{x\in X}\mathcal{E}xt^{1}(\Omega_{X},\mathcal{O}_{X})$ 
which corresponds to $\mathbb{Q}$-Gorenstein deformation (cf., 
Lemma \ref{local no obstruction}, \cite[3.9(i)]{KS}) 
we are done.

Now let us argue the construction $\Kur(X)$. 
It follows from general algebraic deformation theory 
(or the Grauert's construction 
of analytic semi-universal deformation \cite{Grauert}) 
that there exists a formal semi-universal 
family 
$\mathcal{X}\rightarrow {\it Spec}(R)$, where $R$ is the completion of 
an essentially finite type local ring. By using the Grothendieck existence theorem 
\cite{FGA} and 
the Artin algebraicity theorem \cite{Art}, 
we obtain a semi-universal deformation. 
Moreover, in the semi-universal deformation, it follows from \cite[Theorem 3.9(i)]
{KS} that $\mathbb{Q}$-Gorenstein deformation corresponds to one irreducible 
component. 

As we stated, 
we can even take semi-universal $\mathbb{Q}$-Gorenstein deformation space 
$\Kur(X)$ as an $\Aut(X)$-equivariant affine scheme i.e., $\Aut(X)$ acts on 
both the total space of the semi-universal deformation above $\Kur(X)$ and $\Kur(X)$ 
equivariantly while the projection is equivariant.  Indeed, it follows from Luna \'etale slice theorem 
which we apply to Hilbert schemes (see \cite[especially section 2]{AK}) 
combined with \cite[Theorem 3.9(i)]{KS}. 

We remark that for our main applications in this paper we will only need the existence of the versal deformation as an analytic germ, in which case we do not have the action of $\Aut(X)$, but only a holomorphic action of $K$ for a maximal compact subgroup of $\Aut(X)$. This follows from the equivariant version of the construction of Grauert \cite{Grauert} (cf., also \cite{Rim}). 

Now we study a particular example, which we will use in Section \ref{Degree12}.
\begin{exa}
Let $X_1^T$ be the quotient of $\P^2$ by $\Z/9\Z$, where the generator $\xi$ of $\Z/9\Z$ acts by $\zeta_9. [z_1: z_2 : z_3]=[z_1: \zeta_9 z_2: \zeta_9^{-1} z_3]$, and $\zeta_9$ is the primitive ninth root of unity. Then $X_1^T$ is a degree one log Del Pezzo surface, with one $A_8$ singularity at $[1:0:0]$ and two $\frac{1}{9}(1,2)$ singularities at $[0:1:0]$ and $[0:0:1]$. In particular it is $\Q$-Gorenstein smoothable and has Gorenstein index $3$. Note the Fubini-Study metric on $\P^2$ descends to a K\"ahler-Einstein metric. Since this metric has constant positive bisectional curvature, the cohomology group $H^1_{\text{orb}}(X, T_{X})$(the space of harmonic $T_{X}$-valued $(0, 1) $ forms on the orbifold) vanishes (see for example, Proposition 9.4 in \cite{Dem}), so by the obvious orbifold generalization of the Kodaira-Spencer theory $X_1^T$ has no equisingular deformations. 
 By the above general theory  and a dimension counting using the Main Theorem in \cite{Ma},  we have a decomposition  $$\Def(X_1^T)=\Def_1\oplus \Def_2\oplus \Def_3, $$ where $\Def_i$ is the $\Q$-Gorenstein deformation tangent space of the local singularity $p_i$. 
It is not hard to see that the connected component of the 
automorphism group is $\Aut^0(X_1^T)=(\C^*)^2$. We want to  identify its action on $\Def(X_1^T)$. We first choose coordinates on $\Aut^0(X_1^T)$ so that $\lambda=(\lambda_1, \lambda_2)$ acts on $X_1^t=\P^2/(\Z/9\Z)$ by  $\lambda. [z_1:z_2:z_3]=[\lambda_1 z_1:\lambda_2 z_2:z_3]$. Around $p_3$ we may choose affine coordinate $y_1=z_1/z_3$, and $y_2=z_2/z_3$. So the action of $\Z/9\Z$ is given by $\xi. (y_1, y_2)=(\zeta_9  y_1, \zeta_9^2 y_2)$, which is the standard model for the $\frac{1}{9}(1,2)$ singularity. The action of $(\C^*)^2$ is then $\lambda. (y_1, y_2)=(\lambda_1y_1, \lambda_2y_2)$. Now a local deformation of the affine singularity $\frac{1}{9}(1,2)$ can be seen as follows. We embed $\C^2/(\Z/9\Z)$ into $\C^3/(\Z/3\Z)$ by sending $(y_1, y_2)$ to $(u, v, w)=(y_1^3, y_2^3, y_1y_2)$. A versal deformation is given by $uv-w^3=s$. The induced action of $(\C^*)^2$ is then $\lambda. s=\lambda_1^{-3}\lambda_2^{-3} s$.  
This is then the weight of the action on $\Def_3$. Similarly one can see the weight on $\
Def_2$ is given by $\lambda_1^{-3}\lambda_2^{6}$. To see the weight on $\Def_1$, we can embed $X_1^T$ into $\P(1,2,9, 9)$ as a hypersurface $x_3x_4=x_2^9$, by sending $[z_1:z_2:z_3]$ to $[x_1:x_2:x_3:x_4]=[z_1:z_2z_3: z_2^9: z_3^9]$. One can easily write down a space of deformations of $X_1^T$  as $x_3x_4=x_2\Pi_{i=1}^8(x_2+a_ix_i)$. This deformation only partially smoothes the $A_8$ singularity, so that $\Def_1$ can be identified with the space of all vectors $(a_1, \cdots, a_8)$. It is then easy to see the weight of the action of $\lambda$ on $\Def_1$ is $\lambda_1\lambda_2^{-1}$. So we have arrived at:

\begin{lem}\label{local GIT}
The action of $\Aut^0(X_1^T)$ on $\Def(X_1^T)$ is given by 
$$\lambda. (v_1, v_2, v_3)=(\lambda_1\lambda_2^{-1} v_1, \lambda_1^{-3}\lambda_2^{6}v_2, \lambda_1^{-3}\lambda_2^{-3} v_3).$$
\end{lem}
\end{exa}
From Lemma \ref{local-global} we have a linear action of a group $\Aut(X)$ on $\Def(X)$. If $\Aut(X)$ is reductive (for example, when $X$ admits a K\"ahler-Einstein metric, by Matsushima's theorem \cite{Matsu}), one can take a GIT quotient $\Def(X)//\Aut(X)$. 
So we are in the situation of the remark after Lemma 3.6, and locally analytically this can be viewed as a ``local" coarse moduli space of $\Q$-Gorenstein deformations of $X$. The following lemma provides a more precise link between the Gromov-Hausdorff convergence and algebraic geometry.

\begin{lem} \label{continuity}
Let $X_\infty$ be the Gromov-Hausdorff limit of  a sequence of K\"ahler-Einstein Del Pezzo surfaces $X_i$, then for $i$ sufficiently large 
we may represent $X_i$ by  a point $u_i$ in an open neighborhood of the GIT quotient 
$\Kur(X_\infty)//\Aut(X_\infty)$ (analytically interpreted above as $\Def(X)//\Aut(X)$) so that $u_i\rightarrow 0$ as $i$ goes to infinity.
\end{lem} 

\begin{proof}

From Section \ref{DG input}, we know there are integers $m, N$, such that by passing to a subsequence  the surface $X_i$ converges to $X_\infty$, under the projective embedding into  $\P^N$ defined by orthonormal section of $H^0(X_i, -mK_{X_i})$. Since $X_\infty$ has reductive automorphism group, we can choose a Luna slice $S$ in the component of  the Hilbert scheme corresponding to $\Q$-Gorenstein smoothable deformations of $X_\infty$. Hence for $i$ large enough, $X_i$ is isomorphic to a surface parametrized by $s_i(\in S)\to 0$. By the versality, shrinking $S$, we have a map $F\colon S \rightarrow \Kur(X_\infty)$ so that $s$ and $F(s)$ represent isomorphic surfaces. Let $v_i=F(s_i)$. Then $v_i\rightarrow 0$.  Moreover, by  the remark after Lemma \ref{local CM stability}, the corresponding point to $v_i$ in $\Def(X)$  is polystable for $i$ large, thus its image $u_i\in \Kur(X_\infty)//\Aut(X_\infty)$ represents the same surface $X_i$. 
The conclusion then follows. 
\end{proof}

\subsection{Moduli spaces} \label{Moduli space}
In this section we will define precisely what ``moduli of K\"ahler-Einstein 
$\mathbb{Q}$-Fano varieties'' means to us in this paper. 

\begin{defi}[KE moduli stack]\label{KE moduli stack}
We call a moduli algebraic (Artin) stack $\mathcal{M}$ of $\mathbb{Q}$-Gorenstein family of
$\mathbb{Q}$-Fano varieties a \textit{KE moduli stack} if 
\begin{enumerate}
\item It has a categorical moduli $M$ in the category of algebraic spaces; 
\item There is an \'etale covering of $\mathcal{M}$ of the form $\{ [U_i/G_i] \}$ with affine algebraic schemes 
$U_i$ and reductive groups $G_i$, where there is a $G_i$-equivariant 
$\mathbb{Q}$-Gorenstein flat family of $\mathbb{Q}$-Fano varieties. 
\item Closed orbits of $G_i \curvearrowright U_i$  correspond to geometric points of $M$, and parametrize $\Q$-Gorenstein smoothable K\"ahler-Einstein $\mathbb{Q}$-Fano varieties. 
\end{enumerate}
We call the categorical moduli in the category of algebraic space $M$ a \emph{KE moduli space}. 
If it is an algebraic variety, we also call it \emph{KE moduli variety}. 
\end{defi}

\noindent
For an introduction to the theory of algebraic stacks, one may refer to \cite{Andrew}.  
For the general conjecture and for more details on the existence of KE moduli stack, compare Section \ref{K moduli}. 
For our main purposes in proving Theorem \ref{MT}, we only need a much weaker notion. 

\begin{defi}[Analytic moduli space]
An \emph{analytic moduli space} of degree $d$ log Del Pezzo surfaces is a compact analytic space $M_d$ with the following structures:
\begin{enumerate}
\item We assign to each point in $M_d$ a unique isomorphism class  of $\Q$-Gorenstein smoothable degree $d$ log Del Pezzo surfaces. For simplicity of notation, we will denote by $[X]\in M_d$ a point which corresponds to the isomorphism class of the log Del Pezzo surface $X$. 
\item For each $[X]\in M_d$ with $\Aut(X)$ reductive, there is an analytic neighborhood $U$, and a 
quasi-finite locally bi-holomorphic map $\Phi_U$ from $U$ onto an analytic neighborhood of $0\in \Def(X)//\Aut(X)$ 
(where as in the remark after Lemma 3.6, we have chosen a $K$-equivariant identification between analytic neighborhoods in $\Kur(X)$ and $\Def(X)$) 
such that $\Phi_U^{-1}(0)=[X]$ and for any $u\in U$, the surfaces parametrized by $u$ and $\Phi_U(u)$ are isomorphic. 
\end{enumerate}
\end{defi}

\begin{defi}\label{perfect}
We say that an analytic moduli space has {\textit{property (KE)}} if every surface parametrized by  $M_d^{GH}$ is isomorphic to one parametrized by some point in $M_d$.
\end{defi}

\begin{thm}\label{hom}  For any analytic moduli space $M_d$ which has property (KE), there is a homeomorphism from $M_d^{GH}$ to $M_d$, under the obvious map. 
\end{thm}
\begin{proof}
To carry out the strategy described in the introduction, we just need the natural map from $M_d^{GH}$ to $M_d$ to be continuous. It suffices to show that if we have a sequence $[X_i]\in M_d^{0}$ converges to a point $[X_\infty]\in M_d^{GH}$, then $\Phi([X_i])$ converges to $\Phi([X_\infty])$. Unwrapping the definitions, this is exactly  Lemma \ref{continuity}.
\end{proof}

In later sections \ref{Degree34} and \ref{Degree12}, we will construct the analytic moduli space $M_d$ for 
$\mathbb{Q}$-Gorenstein smoothable cases one-by-one. We will show that these $M_d$'s satisfy property (KE). Moreover they are actually categorical moduli of moduli stacks $\M_d$, with a Zariski open subset parametrizing all smooth degree $d$ Del Pezzo surfaces.  Thus Theorem \ref{MT} follows.

\section{The cases of degree four and three}\label{Degree34}

\subsection{Degree four case}\label{Degree4}

In this case  Theorem \ref{MT} has already been proved in \cite{MM}. Following the general strategy outlined in the introduction, we give a partially new proof here.
 Recall that smooth degree $4$ Del Pezzo surfaces are realized by the anti-canonical embedding as intersections of two quadrics in $\P^4$. So in order to construct a moduli space, it is natural to consider the following GIT picture
 $$PGL(5; \C) \curvearrowright H_4={\it Gr}(2,  Sym^2 (\C^5)) \hookrightarrow \P_*
(\Lambda^2 Sym^2 (\C^5))\footnote{In this paper $\P_*(V)=\P(V)$ is the covariant projectivization and $\P^*(V)=\P(V^*)$ is the contravariant projectivization. } ,$$ 
with a linearization induced by the Pl\"ucker embedding. 
Here, ${\it Gr}$ stands for the Grassmanian. 

\begin{thm}[Mabuchi-Mukai\ \cite{MM}] \label{M4} An intersection $X$ of two quadrics in $\P^4$ is
\begin{itemize}
 \item stable $\Longleftrightarrow$ $X$ is smooth;
 \item semistable $\Longleftrightarrow$ $X$  has at worst $A_1$ singularities (nodes);
 \item polystable $\Longleftrightarrow$ the two quadrics are simultaneously diagonalizable, i.e. $X$ is isomorphic to the intersection of quadrics
$$ \,\begin{cases} x_0^2+x_1^2+x_2^2+x_3^2+x_4^2=0\\
 \lambda_0 x_0^2+ \lambda_1x_1^2+\lambda_2x_2^2+\lambda_3x_3^2+\lambda_4x_4^2=0 \end{cases}$$
and no three of the $\lambda_i$s are equal (or equivalently, $X$ is either smooth or has exactly two or four $A_1$ singularities).
\end{itemize}
\end{thm}

Since two  degree four Del Pezzo surfaces as above considered are abstractly biholomorphic if and only if their equations in the  above embeddings are transformed by the natural action of an element of $PGL(5; \C)$ (this follows by the very ampleness of the anticanonical 
line bundle=, the GIT quotient $$M_4:= H_4^{ss} // PGL(5; \C)$$
parametrizes (abstract) isomorphism classes of polystable intersections of two quadrics. We remark here that the same reasoning should be applied later in the other degree cases.

Moreover, since $M_4$ is naturally coarsely isomorphic to the moduli space of binary quintics on $\P^1$, choosing invariants  as in  \cite{Dolga}, Chapter 10.2, we see that  $M_4$ is isomorphic to 
$\P(1,2,3)$ and that  the smooth surfaces are parametrized by the Zariski open subset $M_4^{\rm sm} \cong \P(1,2,3) \setminus D$, where $D$ is an ample divisor cut out by the equation $z_1^2=128z_2$. \\ 

The $d=4$ case of Theorem \ref{MT} then follows from the following:

\begin{thm}\label{deg4.KE.property}
The above constructed $M_4$ is an analytic moduli space with property (KE).
\end{thm}

\begin{proof}
 First we verify that our degenerations of del Pezzo surfaces in $\mathbb{P}^{4}$ can \textit{not} give any  
``pathological test configuration'' in the sense of \cite{LX} 
(called ``almost trivial" test configuration in \cite{Od3}) whose normalization is trivial. 
It is due to the following reason. The central fibers of such pathological test configurations is not equidimensional, so it is
especially \textit{not} Cohen-Macaulay. However, the degenerations in $\mathbb{P}^{4}$ that we considered here
are all Cohen-Macaulay. This is because, in general, weighted projective spaces only have quotient singularities, and so are Cohen-Macaulay 
(cf., e.g., \cite{HR}). Then  the finite times cut by Cartier divisors 
(hypersurfaces) are inductively Cohen-Macaulay (cf., e.g., \cite[page 105]{Mat}). 
Later we will use the similar reasoning for other degrees as well. 

To check $M_4$ is an analytic moduli space, observe that item $(1)$ is obvious, and item $(2)$ follows from the construction of $M_4$ as a GIT quotient (the versal family is the universal one over $H_4$). To see $M_4$ has property (KE), we first use Theorem \ref{logDelPezzo-canonical} to see that any $[X]\in M_4^{GH}$ is parametrized by $H_4$. Then we apply Theorem \ref{KEtoKstability} and Theorem \ref{CM stability}  (since Picard rank of $H_4$ is one, and it is easy to verify the assumptions are satisfied in this case) to see that $[X]$ is parametrized by $M_4$. 
\end{proof} 

  Clearly $\M_4:=[H_4^{ss}/PGL(5;\C)]$ is a quotient stack, so we conclude that it is indeed a KE moduli stack. We make a few remarks here. First of all, the above arguments actually prove that all degree four  K\"ahler-Einstein log Del Pezzo surfaces are parametrized by $M_4$. 
By Theorem \ref{M4} the Gromov-Hausdorff limits of smooth Del Pezzo quartics have only an even number of $A_1$ singularities. The maximum number of such singularities is four. There is exactly one such surface $X_4^T$, which is defined by the equations $x_0x_1=x_2^2=x_3x_4$. It is isomorphic to the quotient $\P^1\times\P^1/(\Z/2\Z)$, where the generator $\xi$ of $\Z/2\Z$ acts as $\xi.(z_1, z_2)=(-z_1, -z_2)$. So it admits an obvious K\"ahler-Einstein metric. 

 It is also easy to see that the action of complex conjugation, which sends a Del Pezzo quartic to its complex conjugate, coincides with the natural complex conjugation on $\P(1,2,3)$.

\subsection{Degree three case}\label{Degree3}
Recall that smooth degree $3$ Del Pezzo surfaces are cubic hypersurfaces in $\P^3$. Note that the anti-canonical bundle is very ample. We recall the following classical GIT picture. The group $PGL(4; \C)$ acts naturally on the space $H_3= \P_{*}(Sym^3 (\C^4)) \cong \P^{19} $ of cubic polynomials. 

\begin{thm}[Hilbert] \label{M3} A cubic surface $X$ in $\P^3$ is
\begin{itemize}
 \item stable $\Longleftrightarrow$ $X$ has at most singularities of type $A_1$;
 \item  semistable $\Longleftrightarrow$ $X$ has at worst  singularities of type $A_1$ or $A_2$;
 \item strictly polystable $\Longleftrightarrow$ $X$ is  isomorphic to the cubic $X_3^T$ defined by equation $x_1x_2x_3=x_0^3$. It is not hard to see that $X_3^T$ has exactly  three $A_2$ singularities, and is isomorphic to the quotient $\P^2/(\Z/3\Z)$, where the generator $\xi$ of $(\Z/3\Z)$ acts by $\xi.[z_1: z_2: z_3]=[z_1: e^{2\pi i/3}z_2: e^{-2\pi i/3}z_3]$. 
\end{itemize}
\end{thm}

Define the quotient stack $\M_3:=[H_3^{ss}/PGL(4;\C)]$ and the corresponding GIT quotient (or in other word, categorical moduli) 
$$M_3:=H_3^{ss}// PGL(4; \C)$$
which parametrizes isomorphism classes of polystable cubics. 
The above Theorem is classical. It was proved by D. Hilbert in his Doctoral dissertation \cite{Hil}. For a modern proof consult \cite{Mum}. 
Moreover, by looking at the ring of invariants 
\cite{Sal}, it is known that $$M_3 \cong \P(1,2,3,4,5),$$
and that $M_3^{\rm sm} \cong \P(1,2,3,4,5) \setminus D$ where $D$ is the ample divisor of equation $(z_1^2-64z_2)^2-2^{11}(8z_4+z_1z_3)=0$. So $M_3^{\rm sm}$  is Zariski open and parametrizes all smooth cubic surfaces. 

Note that we can apply Theorem \ref{CM stability} for universal family 
over $H_3$. Thus it follows that $\M_3$ is a KE moduli stack and $M_3$ is a KE moduli variety. \\

Observe that a Gromov-Hausdorff  limit of smooth K\"ahler-Einstein cubic surfaces has  either exactly three $A_2$ singularities  or at most four $A_1$ singularities. In the former case, it is isomorphic to $X_3^T$. In the latter case, this is the Cayley's cubic $X_3^C$ defined by $x_0x_1x_2+x_1x_2x_3+
 x_2x_3x_0+x_3x_0x_1=0$. It is not hard to see it is isomorphic to the quotient of $X_6/(\Z/2\Z)$, where $X_6$ is the degree six Del Pezzo surface, and the action of $\Z/2\Z$ is induced by the standard Cremona transformation  on $\P^2$, i.e., 
$[z_1: z_2: z_3]\mapsto [z_1^{-1}: z_2^{-1} : z_3^{-1}]$. The existence of K\"ahler-Einstein metrics on $X_3^T$ and $X_3^C$  can also be easily seen using the above quotient description. 

 We remark that it was proved in \cite{DT} that a K\"ahler-Einstein cubic surface must be GIT semistable, and our application of Theorem \ref{CM stability} sharpens this. The existence of K\"ahler-Einstein metrics on cubic surfaces with exactly one $A_1$ singularity was proved by \cite{Wang}, using K\"ahler-Ricci flow on orbifolds and certain calculation of $\alpha$-invariants. In \cite{Spotti2}, by glueing method we know the existence of K\"ahler-Einstein metrics on a partial smoothing of the Cayley cubic $X_3^C$.  For general  cubics with two or three $A_1$ singularities this was previously unknown. Here we actually know that all degree three $\Q$-Gorenstein smoothable K\"ahler-Einstein log Del Pezzo surfaces are parametrized by $M_3$.

 As in the degree four case, the action of complex conjugation on $M_3$ is also given by the natural anti-holomorphic involution.

\section{The cases of degree two and one}\label{Degree12}
\subsection{More detailed study on Gromov-Hausdorff limits}
When the degree is one or two, there are new difficulties as non-canonical singularities could appear in the Gromov-Hausdorff limits. So the  classification of canonical Del Pezzo surfaces (Theorem \ref{logDelPezzo-canonical}) is not enough for our purpose. In degree two by Theorem \ref{Bishop-Gromov} we only need to deal with index 2  log del Pezzo surfaces, which have been classified in \cite{AN}, \cite{Nakayama}, \cite{KK}. We could simply use these classification results directly, but since our assumption is much more restricted, we provide a more elementary approach which treats both $d=1$ and $d=2$ cases.

A common feature of the two cases is the existence of a holomorphic involution.  For a  degree two Del Pezzo surface $X$, it is well-known that the anti-canonical map defines a double cover of $X$ to $\P^2$. Therefore $X$ admits an involution $\sigma$ (``Geiser involution") which is simply the deck transformation of the covering map. The fixed locus of $\sigma$ is smooth quartic curve. If $X$ admits a K\"ahler-Einstein metric $\omega$, then by \cite{BM} $\omega$ must be invariant under any such $\sigma$. 
Similarly for a  degree one Del Pezzo surface $X$, the linear system $|-2K_X|$ defines a double cover of $X$ to $\P(1,1,2)\subset \P^3$. So $X$ also admits an involution $\sigma$ (``Bertini involution").   Again any such $\sigma$ must preserve the K\"ahler-Einstein metric if $X$ admits one. The fixed locus of $\sigma$ consists of the point $[0:0:1]$ and a sextic in $\P(1,1,2)$.

\begin{lem}
Suppose a sequence of degree one (or two) K\"ahler-Einstein Del Pezzo surfaces $(X_i, \omega_i, J_i)$ converges to a Gromov-Hausdorff limit $(X_\infty, \omega_\infty, J_\infty)$, then  by passing to a subsequence one can take a limit $\sigma_\infty$, which is a holomorphic involution on $X_\infty$. 
\end{lem}

\begin{proof} This is certainly well-known. We include a proof here for the convenience of readers. Let $p_1, \cdots, p_n$ be the singular points of $X_\infty$. We denote $\Omega_r=X_\infty\setminus \cup_{j=1}^n B(p_j, r)$. For any $r>0$ small, from the convergence theorem \ref{orbifold compactness},  we know that for $i$ sufficiently large, there are $\sigma_i$ invariant open subsets $\Omega_{i} \subset X_i$ and embeddings $f_i: \Omega_i\rightarrow X_\infty\setminus \{p_1,\cdots, p_n\}$ such that $\Omega_r$ is contained in the image of each $f_i$ and $(f_i^{-1})^*(\omega_i, J_i)$ converges to $(\omega_\infty, J_\infty)$ smoothly.  Then, by passing to a subsequence, the isometries $(f_i^{-1})^*\sigma_i$ converge to a limit $\sigma_{r, \infty}: \Omega_r \rightarrow X_\infty$ with $\sigma_{r, \infty}^*(\omega_\infty, J_\infty)=(\omega_\infty, J_\infty)$. Then we can let $r$ tend to zero and choose a diagonal subsequence so that $\sigma_{r, \infty}$ converges to a holomorphic isometry $\sigma_\infty$ on 
$X_\infty\setminus \{p_1,\cdots, p_n\}$. Then by the Hartog's extension theorem, $\sigma_\infty$ extends to a holomorphic isometry on the whole $X_\infty$. It is also clear $\sigma_{\infty}^2$ is the identity.
\end{proof}
 
\begin{thm} \label{double cover classification}
In the degree two case, $X_\infty$ is either a double cover of $\P^2$ branched along a quartic curve, or a double cover of $\P(1,1,4)$ branched along a degree 8 curve not passing through the vertex $[0:0:1]$.
In the degree one case, $X_\infty$ is either a double cover of $\P(1,1,2)$ branched along the point $[0:0:1]$ and a sextic , or a double cover of $\P(1,2,9)$ branched along the point $[0:1:0]$ and a degree 18 curve  not passing through the vertex $[0:0:1]$.
\end{thm}
\begin{proof}
We first treat the case of degree one. The proof of the degree two case is essentially the same and we will add some remarks later.  Denote by $Y_i$ the quotient of $X_i$ by $\sigma_i$, so the quotient $Y_\infty=X_\infty/\sigma_\infty$ is the Gromov-Hausdorff limit of $Y_i$'s. For each integer $m$ we have an orthogonal decomposition $H^0(X_i, -mK_{X_i})=V_i\oplus W_i$ with $V_i$ being the $+1$ eigenspace and $W_i$ the $-1$ eigenspace. Then we have a corresponding decomposition $H^0(X_\infty, -mK_{X_\infty})=V_\infty\oplus W_\infty$ on $X_\infty$. Now, by  constructing orthonormal $\sigma_\infty$-invariant sections of $-kK_{X_\infty}$ for some $k$ large divisible, one can show that there is a well-defined map $\iota_\infty: X_\infty\rightarrow \P^*(V_\infty)$, which induces a projective embedding of $Y_\infty$. By an adaption of the H\"ormander technique (\cite{Tian1}, \cite{DS}), this implies that the orthonormal $\sigma_i$-invariant sections of $-kK_{X_i}$(equivalent to sections of $-l K_{Y_i}$ for some 
integer $l$) define an embedding (Tian's embedding) of $Y_i$ into $\P^*(V_i)$ for $i$ sufficiently large. Moreover,  we may assume $Y_i$ converges to $Y_\infty$ as normal varieties in $\P^N$ for some integer $N$. Since $Y_i$'s are all isomorphic to $\P(1,1,2)$, we see that $Y_\infty$ is $\Q$-Gorenstein smoothable, and there is a partial $\Q$-Gorenstein smoothing of $Y_\infty$ to $\P(1,1,2)$. 

\begin{clm}
$Y_\infty$ is isomorphic to $\P(1,2,9)$. 
\end{clm}

Since the degree is preserved in the limit, we have $K^2_{Y_\infty}=8$ and thus we may apply \cite{HP}. Notice the full proof in \cite{HP} relies on the classification theorem of Alexeev-Nikulin \cite{AN}, but in our case we only need the more elementary part \cite{HP1}, without use of \cite{AN}. So we know $Y_\infty$ is either a toric log Del Pezzo surface $\P(a^2, b^2, 2c^2)$ with $a^2+b^2+2c^2=4abc$ or its partial smoothings. Since the orbifold structure group of $X_\infty$ always has order less than $12$,  the order of all the orbifold structure groups of $Y_\infty$ must be less or equal to $22$. Then, by an easy investigation of the above Markov equation, we see that $Y_\infty$ must have two singularities, one of type $A_1$ and one of type $\frac{1}{9}(1,2)$. It could be possible that $Y_\infty$ is a partial smoothing of $\P(9, b^2, 2c^2)$, but we claim then it must be $\P(1,2, 9)$. For this we need to go back to the proof in \cite{HP1}. For the minimal resolution $\pi: \tY_\infty\rightarrow Y_\infty$, 
let $n$ be the largest number such that there is a birational morphism $\mu_n$ from $\tY_\infty$ to  the $n$-th Hirzebruch surface $\F_n$. Let $B'$ be the proper transform of the negative section $B$ in $\F_n$, and let $p: \tY_\infty\rightarrow\P^1$ be the composition of $\mu_n$ with the projection map on $\F_n$. Then by a theorem of Manetti \cite[Theorem 11]{Ma} (see also \cite[Theorem 5.1]{HP1}) we know that $n\geq 2$, and the exceptional locus $E$ of $\pi$ is the union of $B'$ and the components of degenerate fibers of $p$ with self-intersection at most $-2$; furthermore, each degenerate fiber of $p$ contains a unique $-1$ curve. Moreover, by the proof of Theorem 18 in \cite{Ma} (see laso Theorem 5.7 in \cite{HP1}), there are only two possible types for the dual diagram of  the degenerate fiber: one type is that two strings of curves of self-intersection at most $-2$ joined by a $(-1)$-curve, and the other type is that we join a string of $(-2)$-curves though a $(-1)$ curve to the middle of a string of 
curves of self-intersection at most $-2$. In our case we know $Y_\infty$ has exactly one $A_1$ and one $\frac{1}{9}(1,2)$ singularity. By general theory on the resolution of cyclic quotient singularities we know $E$ is the disjoint union of a (-2)-curve and a string of a (-2)-curve and a (-5)-curve. Then one easily sees that the only possibility is that there is exactly one degenerate fiber of $\tY_\infty$ which consists of a string of (-2)-(-1)-(-2)-curve, and one of the (-2)-curves in the string intersects the horizontal section $B'$ which is a (-5)-curve. Clearly $\tY_\infty$ is then a toric blown-up of $\F_5$ and then $Y_\infty$ is toric which must be $\P(1,2,9)$.   This completes the proof of the claim.

The degree of the branched locus follows from the Hurwitz formula for coverings.   The degree 18 curve can not pass through the point $[0:0:1]$, for otherwise the equation would be $a _0x_3f_9(x_1,x_2)+a_1f_{18}(x_1,x_2)=0$. Then by Lemma \ref{quotient singularity} below the singularity on the branched cover is not quotient, so can not be $X_\infty$ by Theorem \ref{orbifold compactness}. This finishes the proof of Theorem \ref{double cover classification} for degree one case.

In the degree two case we can follow exactly the same arguments, noticing that, in this case, $Y_\infty$ must have degree equal to $9$ and that the associated  Markov type equation to be satisfied is now $a^2+b^2+c^2=3abc$ (corresponding to the weighted projective space $\P(a^2,b^2,c^2)$). Thus it follows, by inspection as above, that the only possibility is that $Y_\infty=\P(1,1,4)$ (since  the standard projective plane is the only $\Q$-Gorenstein smoothing of the above weighted projective space). \end{proof}

Noticing that in the above situation our double cover can be realized as a hypersurface in a one dimensional higher weighted projective space, in terms of equations we have the following:
\begin{cor}\label{equation degree one}
In degree one case $X_\infty$ is either a sextic hypersurface in $\P(1,1, 2, 3)$ of the form  $x_4^2=f_6(x_1, x_2, x_3)$, or a degree 18 hypersurface in $\P(1,2,9,9)$ of the form $
x_3^2+x_4^2=f_{18}(x_1, x_2)$. 
\end{cor}

\begin{cor} \label{equation degree two}
In degree two case $X_\infty$ is either a quartic hypersurface in $\P(1,1,1,2)$ of the form $x_4^2=f_4(x_1, x_2, x_3)$ or an octic hypersurface in $\P(1,1,4,4)$ of the form $x_3^2+x_4^2=f_8(x_1, x_2)$.
\end{cor}

\begin{lem} \label{quotient singularity}
Suppose $f$ is a polynomial and the surface $w^2=f(x, y)$ in $\C^3$ or its $\Z/2\Z$ quotient by $(x,y,w)\mapsto (-x, -y, -w)$ has a quotient singularity at the origin, then $f$ must contain a monomial with degree at most three. 
\end{lem}
\begin{proof}
If the singularity is a quotient singularity, then singularity $w^2=f(x,y)$ in $\C^3$ is canonical  since the finite map does not have 
branch divisor. Then the statement follows from a criterion of canonicity in terms of Newton polygon (cf. e.g., \cite[Corollary 1.7]{Ish}). 

\end{proof}
The idea of using involutions to study $X_\infty$ was previously used in \cite{Tian1}, where some partial  results were claimed. For example, in Proposition 6.1 of \cite{Tian1}, it was stated that in degree two case $X_\infty$ can have at most $\frac{1}{4}(1,1)$ singularities besides canonical singularities and $|-2K_{X_\infty}|$ is base point free. This agrees with the above result. But, as one can see from the following example, the claims in Proposition 6.2 of \cite{Tian1} that in degree one case $X_\infty$ can have at most one non-canonical singularity and $|-2K_{X_\infty}|$ is base point free, are both incorrect. 

Now we show explicit examples of K\"ahler-Einstein log Del Pezzo surfaces with non-canonical singularities in both degree one and two \footnote{We are indebted to A. Kasprzyk for discussions related to these examples  \cite{KKL}.} . In the next two subsections it will be proved that both are parametrized in the moduli spaces.

\begin{exa} \label{degree2 toric}
Let $X_2^T$ be the quotient of $\P^1\times\P^1$ by the action of $\Z/4\Z$, where the generator $\xi$ of $\Z/4\Z$ acts by $\xi.([z_1:z_2], [w_1:w_2])=([\sqrt{-1}z_1: z_2], [-\sqrt{-1}w_1: w_2])$.  Then it is easy to see that $X_2^T$ is a degree two log Del Pezzo surface, with two $A_3$ singularities and two $\frac{1}{4}(1,1)$ singularities. The standard product of round metrics on $\P^1\times\P^1$ descends to a K\"ahler-Einstein metric on $X_2^T$.  The space $H^0(X_2^T, -K_{X_2^T})$ is spanned by the sections $z_1^2w_1^2$, $z_2^2w_2^2$, and $z_1z_2w_1w_2$. So a generic divisor in $|-K_{X_2^T}|$ is given by the union of two curves $z_1w_1+az_2w_2=0$ and $z_1w_1+bz_2w_2=0$ for $a\neq b$, and is thus reducible. The space $H^0(X_2^T, -2K_{X_2^T})$ is spanned by sections $z_1^4w_1^4, z_2^4w_2^4, z_1^2z_2^2w_1^2w_2^2, z_1^3z_2w_1^3w_2, z_1z_2^3w_1w_2^3, z_1^4w_2^4, z_2^4w_1^4$. 
The subspace $U$ spanned by the first five sections is generated by  $H^0(X_2^T, -K_{X_2^T})$.  
The involution $\sigma$ maps $([z_1:z_2], [w_1:w_2])$ to $([w_1: w_2], [z_1: z_2])$.   The $+1$ eigenspace $V_1$ is still six dimensional, spanned by $U$ and the element $z_1^4w_2^4+z_2^4w_1^4$.  It is easy to see that the image of $X$ under the projection to $V$ is the cone over the rational normal curve of degree $4$, i.e. $\P(1,1,4)$. The branch locus is defined by $z_1^4w_2^4=z_2^4w_1^4$, with singularies exactly at the two $A_3$ singularities. We can also see directly that $X_2^T$ is the hypersurface in $\P(1,1,4,4)$ defined by $x_1^4x_2^4=x_3x_4$. The map is given by 
$$([z_1: z_2], [w_1: w_2])\mapsto (z_1w_1, z_2w_2, z_1^4w_2^4, z_2^4w_1^4). $$
 Make a change of variable $x_3'=x_3+x_4$ and $x_4'=x_3-x_4$, then the projection to the $(x_1, x_2, x_3')$ plane realizes $X_2^T$ as a double cover of $\P(1,1,4)$. 
 \end{exa}

\begin{exa} \label{degree1 toric}
Let $X_1^T$ be the example studied in Section \ref{AG input}. It is a toric degree one K\"ahler-Einstein log Del Pezzo surface with one $A_8$ singularity and two $\frac{1}{9}(1,2)$ singularities. It can be viewed as a hypersurface in $\P(1,2, 9,9)$ given by the equation $x_3x_4=x_2^9$. The embedding is defined by 
$$[z_1:z_2:z_3]\mapsto [z_1: z_2z_3: z_2^9: z_3^9].$$ 
The projection map $\P(1, 2, 9, 9)\rightarrow \P(1,2,9)$ sending $[x_1:x_2:x_3:x_4]$ to $[x_1: x_2: x_3+x_4]$ realizes $X_1^T$ as the double cover of $\P(1, 2, 9)$, branched along the rational curve $x_2^9=x_3^2$. 
On $X_1^T$ the holomorphic involution $\sigma$ simply exchanges $z_2$ with $z_3$. 
One can see the pluri-anti-canonial linear systems on $X_1^T$. $H^0(X_1^T, -K_{X_1^T})$ is spanned by $z_1^3$ and $z_1z_2z_3$, so it has a fixed component $z_1=0$. $H^0(X_1^T, -2K_{X_1^T})$ is spanned by $z_1^6, z_1^4z_2z_3, z_1^2z_2^2z_3^2, z_2^3z_3^3$, so it has two base points $[0:1:0]$ and $[0:0:1]$. We will show below that $X_1^T$ is the Gromov-Hausdorff limit of a sequence of K\"ahler-Einstein degree one Del Pezzo surfaces. This implies that the Proposition 6.2 in \cite{Tian1} is incorrect. Similarly it is easy to see that $|-3K_{X_1^T}|$ is base point free.  
As before we have an eigenspace decomposition $H^0(X_1^T, -mK_{X_1^T})=V_m\oplus W_m$ for $\sigma$. Then $|V_6|$ is base point free, and it defines the embedding of $\P(1,2,9)$ into $\P^{15}$ by sections of $\O(18)$.
\end{exa}

\subsection{Degree two case} \label{degree two case}

 We first recall the moduli space constructed in  \cite{Mukai}.  For a smooth Del Pezzo surface $X$ of degree 2 the anti-canonical map  is a double covering to $\P^2$ branched along a smooth quartic curve $F_4$.  The geometric invariant theory for quartic curves is well-understood 
(cf. \cite{Mum}) as follows. (Note that Mukai's citation \cite[9.3]{Mukai} misses one case.) 

\begin{lem}[\cite{HL} Theorem 2]
For a quartic curve $F_4$ in $\P^2$ we have: 
\begin{itemize}
\item $F_4$ is stable if and only if $F_4$ has only rational double points of type $A_1$ or $A_2$;
\item $F_4$ is strictly polystable  if and only if $F_4$ is one of the following: either a double conic or 
a union of two reduced conics that are tangential at two points and at least one is smooth (called cateye and ox in \cite{HL}). 
\end{itemize}
\end{lem}

\noindent
It follows that the quotient $Q:=\mathbb{P}_*(Sym^4\C^3)^{ss}//PGL(3;\C)$
parametrizes certain canonical log Del Pezzo surfaces of degree $2$, away from the double conic.  The stable curves parametrize surfaces with at worst $A_1$ or $A_2$ singularities, the double conic parametrize a non-normal surface with non orbifold singularities (note in fact that the variety of equation $t^2=(x^2+y^2+z^2)^2$ can be decomposed into irreducible factors $(t+x^2+y^2+z^2)(t-x^2+y^2+z^2)=0$), and the other polystable curves parametrize surfaces with exactly $2A_3$ singularities.
As in \cite{Mukai}, we blow up the point corresponding to the double conic to obtain a new variety, denoted by $M_2$. Let $E$ be the exceptional divisor. Then, as in \cite{Shah}, we know $E$ is isomorphic to the GIT moduli space 
$\mathbb{P}_*(Sym^8\C^2)^{ss}//PGL(2;\C)$, parametrizing binary octics  $f_8(x, y)$. Moreover,

\begin{thm}
$M_2$ is an analytic moduli space of log Del Pezzo surfaces of degree two. For any $[s]\notin E$, $X_s$  is the double cover of $\P^2$ branched along the polystable quartic defined by $[s]$, and for $[s]\in E$,  $X_s$ is the double cover of $\P(1,1,4)$ (i.e. the cone over the rational normal curve in $\P^5$) branched along the hyperelliptic curve $z_3^2=f_8(z_1, z_2)$, where $f_8$ is the polystable binary octic defined by $[s]$. 
\end{thm}

The proof uses some ideas of \cite{Shah} as written in \cite{Mukai}, 
but note that the proof in  \cite{Shah}  is incomplete about the existence of moduli algebraic stack nor 
the blow up is its coarse moduli scheme, since no family has been constructed.  
 The argument in \cite{Shah} 
is curve-wise and only verifies the properness criterion formally. 

\begin{proof}
Let  $H_4$ be the Hilbert scheme of quartics in $\P^2$, and fix a non-degenerate conic $C=\{q=0\}$. 
We identify the automorphism group of $C$ with $PGL(2;\C)$
(The notation $PGL(2;\C)$ only appears in this context of this proof, so should not be confusing).
Denote by $\Psi$ the (9-dimensional) $PGL(2;\C)$-invariant subspace of  $H^0(\mathbb{P}^2, \mathcal{O}(4))$ that corresponds to $H^0(C,\mathcal{O}(4)|_{C})$.
Take an affine space $\mathbb{A}\simeq \C^9$ in $H_4$ which represents $\{ q^2+f_4(x, y, z)\}$ for all quartics $f_4 \in \Psi$. From the construction, this gives 
a Luna \'{e}tale slice. Note that the blow up $\B$ of $\mathbb{A}$ at $0$ is a closed subvariety of $\mathbb{A}\times \mathbb{P}_*(\mathbb{A})$, and let $\mathbb E$ be its exceptional divisor. Let $\mathcal{B}\subset \mathbb{A}\times (\mathbb{A}\setminus \{0\})$ be the cone over $\B$, and $\mathcal{E}=\{0\}\times (\mathbb{A}\setminus \{0\})$ be the cone over $ \mathbb  E$.
For each point  $(a, b) \in \mathcal{B}$, we can associate 
the curve $q^2+b=0$ in $\mathbb{P}^2$. These form a flat projective family $\mathcal Q$ over $\mathcal{B}$. 

On the other hand, consider the trivial family of $(\mathbb{P}^2, C)$ over $\mathcal{B}$. 
We blow up 
$C\times \mathcal{E}$ and contract the strict transform of 
$\mathbb{P}^2\times \mathcal{E}$. It is possible because $\mathcal{E}$ is a 
Cartier divisor in $\mathcal{B}$ and the classical degeneration (deformation to the normal cone of $C$) of $\mathbb{P}^2$ 
to $\mathbb{P}(1,1,4)$ over a smooth curve is constructed in the same way, so
we can do it locally and glue the contraction morphism. Denote the family constructed in this way by $\mathcal{P}\rightarrow \mathcal{B}$.
The generic fibers are $\mathbb{P}^2$ and special fibers (those over $\mathcal E$) are $\mathbb{P}(1,1,4)$. 
We also obtain a natural family of conics $C_\mathcal{P}\subset \mathcal{P}$ over $\mathcal{B}$. 

All the above process is $PGL(2;\C)\times \mathbb{C}^*$-equivariant. Thus we can construct $PGL(2;\C)$-invariant complement of $\mathbb{C}q^2$ 
in $H^0(\mathcal{P}_u, \mathcal{O}(2C_\mathcal{P}))$ ($u\in \mathcal B$) in a continuous way, and extend  the family of quartics $\mathcal Q|_{(\mathcal B \setminus \mathcal E)}$
  to the whole $\mathcal{B}$. We denote the new  total space  by 
$\mathcal{D}$. Notice that over $\mathbb  E$ this is a family of binary octics. Then construct $\mathcal{S}$ as the double of 
$\mathcal{P}$ branched along $\mathcal{D}$. As everything is again 
$PGL(2;\C)\times \mathbb{C}^*$-equivariant, we can first divide by $\mathbb{C}^*$  and obtain
 a $\mathbb{Q}$-Gorenstein flat family $S$ of degree two 
log Del Pezzo surfaces over $\B$. 

There is still an action of $PGL(2;\C)$ on $\B$. We consider GIT with respect to this action and 
with $PGL(2; \C)$-linearized line bundle $\mathcal{O}_\B(-\mathbb  E)$.
The natural morphism $\B^{ss}//PGL(2;\C)\rightarrow \mathbb{A}//PGL(2;\C)$ is 
an isomorphic away from $\mathbb E\subset B$ and $0\in \mathbb{A}$. So this is a blow up with exceptional divisor $\mathbb E^{ss}//PGL(2;\C)$. 
By the  local picture of GIT (\cite[Prop 5.1]{Shah}), 
we can see that $\mathbb{A}//PGL(2;\C)\rightarrow H_4^{ss}//PGL(3;\C)$ is \'etale 
(or in differential geometric language, local bi-holomorphism) 
around $0$. This follows completely the same way as in \cite[Prop 5.1]{Shah} 
or the proof of famous Luna \'etale slice theorem.
Hence, the blow up $\B^{ss}//PGL(2;\C)\rightarrow \A//PGL(2;\C)$ induces blow up $M_2$ of $H_4//PGL(3;\C)$.

To see that $M_2$ is an analytic moduli space for degree two log Del Pezzo surfaces, we only need to check the item (2) in the definition. For this, one  simply notices that, by construction, for any $[s]\in M_2$ there is a Luna's slice $V$ in $H_4$ or in $\B$ (depending on whether $[s]$ is in $E$ or not).  Then by versality there is an $\Aut(X_{s})$ equivariant analytic map $\Psi_U$ from a small analytic neighborhood $U=V//\Aut(X_{s})$ of $[s]$ to the GIT quotient  $\Kur(X_s)//\Aut(X_s)$ so that $\Phi_U^{-1}(0)=0$. Then it follows that $\Psi_U$ is a finite map onto an open neighborhood of $0$.  

In terms of \'etale topology one can also directly check the versality by going through our construction.
We only need to check our 
$(H_4^{ss}\setminus PGL(3;\C)q^2) \coprod \B^{ss}$ is versal in \'etale 
 topology. That is, 
given a $\mathbb{Q}$-Gorenstein 
projective family $f\colon \mathcal{X}\to S$ of our log del Pezzo surfaces of degree $2$, there is a morphism 
$\tilde{S}\rightarrow (H_4^{ss}\setminus PGL(3;\C)q^2) \coprod \B^{ss}$ 
compatible with fibers where $\tilde{S}\rightarrow S$ is an \'etale cover. For this, we can first construct a degenerating family of $\mathbb{P}^2$ to 
$\mathbb{P}(1,1,4)$ over $S$ and from the $\mathbb{Q}$-Gorenstein deformation theory of $\mathbb{P}(1,1,4)$ (with $1$-dimensional smooth semi-universal deformation space)  we know that the locus of $\mathbb{P}(1,1,4)$ should be a Cartier divisor so that we can convert the process to obtain a family of reduced quartics of $\mathbb{P}^2$. 
Thus we have a compatible morphism to 
$(H_4^{ss}\setminus PGL(3;\C)q^2) \coprod B^{ss}$ 
locally in  \'etale topological sense.  
\end{proof}

\begin{rmk}
In terms of algebro-geometric language, $M_2$ coarsely represents the algebraic stack $\mathcal{M}_2$ constructed by gluing together the quotient stacks $[\B^{ss}/PGL(2;\C)]$ and  $[(H_4^{ss}\setminus PGL(3;\C)q^2)/PGL(3;\C)]$. 
\end{rmk}
\begin{rmk}
Replacing blow up and its cone as above by weighted blow up and its quasi-cone, 
the argument in \cite{Shah} can be completed to prove that the blow up is a coarse moduli 
scheme of degree two K3 surfaces and  its degenerations.  
\end{rmk}

The proof of Theorem \ref{MT} follows from the fact that all smooth degree two Del Pezzo surfaces are parametrized by $M_2$ and by

\begin{thm} \label{degree two perfect}
$M_2$ has property (KE). 
\end{thm}
\begin{proof} By Theorem \ref{double cover classification} there are two possibilities for $X\in M_2^{GH}$: it is either a double cover of $\P^2$ branched along a quartic $f_4(x_1,x_2,x_3)=0$, or a double cover of $\P(1,1,4)$ branched along a hyperelliptic octic curve $x_3^2-f_8(x_1,x_2)=0$. It suffices to show $f_4$ and $f_8$ are polystable. For this we use Theorem \ref{KEtoKstability} and Theorem \ref{CM stability}. When applying Theorem \ref{CM stability}, in the first case we choose  $S=\P_*({\it Sym}^4(\C^3))$; in the second case we choose $S=\P_*(Sym^8(\C^2))$. Note that both these parameter spaces have Picard rank one. Also recall the first paragraph of the proof of (\ref{deg4.KE.property}) 
which asserts that there is no pathological test configurations in the sense of \cite{LX} in this situation. 
Thus we can apply theorem \ref{hom} and conclude that all the points in this glued moduli are indeed GH limits.
\end{proof}

So we also conclude that $\M_2$ is a KE moduli stack. As it is immediately clear from the proof, the complex conjugation acts on $M_2$ by the natural anti-holomorphic involution. 

\begin{rmk} \label{Tian conjecture}
 In \cite{Tian1} it is conjectured that degenerations of K\"ahler-Einstein Del Pezzo surfaces should have canonical singularities. In this section we have seen that this conjecture is in general false, as all the surfaces parametrized by the exceptional divisor $E$ have exactly two non-canonical singularities of type $\frac{1}{4}(1,1)$. In general dimension one expects the compact moduli space of smoothable $\Q$-Fano varieties to have log terminal singularities, see \cite{DS}. This type of singularities also appear to be  the worst singularities allowed for K-semistability of Fano varieties, see \cite{Od2}.
\end{rmk}

We finish this subsection by a discussion on the surfaces parametrized by the ox and cateyes, which will be used in our study of degree one case. 
These are defined by equations in $\P(1,1,1,2)$ parametrized by $\lambda=[\lambda_1:\lambda_2]$ in $(\P^1\setminus \{[1:1]\})$ which we denotes by $X_2^{\lambda}$.  The equation of $X_2^{\lambda}$ is 
$$w^2=(\lambda_1 z^2+xy)(\lambda_2 z^2+xy). $$
It is clear that when we interchange $\lambda_1$ and $\lambda_2$ we get isomorphic surfaces. When $\lambda$ is $[1:0]$ or $[0:1]$, the branch locus is an \emph{ox} and the surface $X_2^\infty=X_2^{\lambda}$ with exaclty two $A_3$ plus one $A_1$ singularities, otherwise the branch locus is a \emph{cateye} and $X_2^\lambda$ with exactly two $A_3$ singularities.  By Theorem \ref{degree two perfect} this family of surfaces all admit K\"ahler-Einstein metrics. As $\lambda$ tends to $[1:1]$ these K\"ahler-Einstein surfaces converge to $X_2^T$, with the obvious K\"ahler-Einstein metric.

One can see that $X_2^\infty$ is  a global quotient of $\P^1\times \P^1$, as follows. Consider the action of $\Z/4\Z$ on $\P^1\times \P^1$, where the generator $\xi$ acts by
\[
\xi. ([z_1: z_2], [w_1:w_2])=([-w_1:w_2], [z_1:z_2]).
\]
Then there are exactly four points with nontrivial isotropy. Let $Y$ be the quotient. Then the points $([0:1], [0:1])$ and $([1:0], [1:0])$ are $A_3$ singularities and $([1:0], [0:1])$ 
and $([0:1], [1:0])$ are $A_1$ singularities. 
One can see that the anti-canonical map $p$ from $Y$ to $\P^2$ is given by 
$$([z_1:z_2], [w_1:w_2])\mapsto (z_1^2w_1^2: z_2^2w_2^2: z_1^2w_2^2+z_2^2w_1^2), $$
and the corresponding involution to the double covering structure 
is $$\sigma. ([z_1:z_2], [w_1:w_2])=([w_1:w_2], [z_1:z_2]). $$
The branch locus is defined by $xy(z^2-4xy)=0$ in $\P^2$, which is isomorphic to the ox. So $Y$ is exactly $X_2^\infty$, and it admits an explicit K\"ahler-Einstein metric. 

Notice that $\P^1\times \P^1$ or $\P^2$ has no deformations, so  their quotients by any finite group have no equisingular deformations. But clearly for $\lambda\neq [1:0], [0:1]$,  $X_2^\lambda$ has nontrivial equisingular deformations, so it can not be a global quotient of $\P^2$ or $\P^1\times \P^1$. 

\subsubsection{Relation with moduli of curves}\label{curve.2}

Naturally considering  the associated branch locus for each double cover (i.e. the bi-anti-canonical map), we can regard our moduli $M_2$ as the GIT moduli of bi-canonically 
embedded Hilbert polystable genus $3$ curves, which is constructed in \cite{HL}. 
Indeed, by a direct comparison, 
the corresponding set of parametrized curves are the same. 
We have a $1$-dimensional tacnodal curves and $5$-dimensional hyperelliptic curves. 
They intersect at one point corresponding to the curve $z^2=x^4y^4$ in $\mathbb{P}(1,1,4)$. 
From this point of view, the proof that the moduli space is a blow up of the GIT moduli of plane quartics is given in \cite{Arte} due to David Hyeon. Our proof recovers this result, 
modulo the criterion of the Hilbert stability. 

Thus a natural question would be the corresponding ``Del Pezzo surface 
modular interpretation" for the flipped contraction which contracts the tacnodal 
locus in the paper \cite{HL}. In general, we can ask: 

\begin{quest}
What are the modular interpretations {\textit{via log Del Pezzo surfaces}} 
for each step of the Hassett-Keel program in \cite{HL}? In addition, are there  also
stability interpretations for them? 
\end{quest}

\subsection{Degree one case}\label{Degree1}

From Section \ref{DG input} we know that for any $X\in M_1^{GH}$, there are only three possible types for the non-canonical singularities.  Moreover, we have:

\begin{lem}
The canonical singularities in $X\in M_1^{GH}$ are either $A_1, \cdots, A_8$ or 
$D_4$. 
\end{lem}
\begin{proof} This follows from Theorem \ref{Bishop-Gromov} and the Noether formula for singular surfaces [\cite{HP}, Proposition 2.6]
$$
\rho(X)+K_X^2+\sum _{P\in Sing(X)} \mu_P =12\chi(\O_X)-2, 
$$
where $\rho(X)$ is the Picard rank of $X$ and $\mu_P$ denotes the Milnor number. Notice that $\chi(\O_X)=1$ by the Kodaira vanishing theorem and that the Milnor number of an $A_k$, $D_k$ or $E_k$ singularity is $k$. 
\end{proof}

We mention that,  by using the K\"ahler-Ricci flow and calculating certain $\alpha$-invariant, it has been proved in \cite{Wang},  \cite{CK} that a degree $1$ log Del Pezzo surface with only $A_n$ singularities admits a K\"ahler-Einstein metric, if $n\leq 6$.

\subsubsection{First step: GIT}\label{dP1.Gore}
By Corollary \ref{equation degree one},  a Gromov-Hausdorff limit in degree one is either a double cover of $\P(1,1,2)$ branched along a sextic or a double cover of $\P(1,2, 9)$ branched along a degree 18 curve. As the first step, we will  construct a moduli space of surfaces that are double cover of $\P(1,1,2)$ branched along a sextic that does not pass through $[0:0:1]$. 
These surfaces have equations  $w^2=F(x,y,z) \subset \mathbb{P}(1,1,2,3)$, where $F$ contains a nonzero term $z^3$.

Although the automorphism 
group of  $\mathbb{P}(1,1,2)$ in non-reductive, 
we can construct a compact moduli space of such sextics in 
$\mathbb{P}(1,1,2)$ which are polystable in appropriate GIT sense, following \cite{Shah}. 
Instead of the honest automorphism group $Aut(\mathbb{P}(1,1,2))$, we consider the action of $SL(2;\C) \ltimes H^0(\mathbb{P}^1,\mathcal{O}(2))$ 
which is a finite cover of $Aut(\mathbb{P}(1,1,2))$ and a subgroup of 
$Aut(\mathbb{P}(1,1,2),\mathcal{O}(2))$ (i.e. it also acts on the
linearization). 
First we fix the translation action of $H^0(\mathbb{P}^1,\mathcal{O}(2))$ by requiring the vanishing of the coefficient of $z^2$. Thus we only need to consider surfaces of the form $$w^2=z^3+f_4(x,y)z+f_6(x, y).$$ 
Then, by dividing out by the natural $\mathbb{C}^*$-action 
on $f_4$ and $f_6$ with weights $4, 6$ repectively,  
we obtain a weighted projective space $\P_s:=\mathbb{P}(2,2,2,2,2,3,3,3,3,3,3,3)$ 
as a parameter space. What is left is the action of $SL(2;\C)$ in the two variables $x, y$. Thus we get a GIT quotient  
$$M_1':=\P_s^{ss}//SL(2;\C)$$ as 
a moduli space.  This is similar to \cite{Shah}, where the GIT of degree $12$ curves in $\P(1,1,4)$ was studied.  We have the following  classification of singularities for polystable locus (compare \cite{Shah}, Theorem 4.3): 

\begin{lem} \label{degree one stable classification}
With respect to the GIT stability of the above $SL(2;\C)$-action, our surface $[w^2=z^3+zf_4(x,y)+f_6(x, y) \subset \mathbb{P}(1,1,2,3)]$ 
is: 
\begin{enumerate}
\item stable if and only if it contains at worst $A_k$ singularities;
\item strictly polystable if and only if it contains exactly two $D_4$ singularities or 
$SL(2;\C)$-equivalent to $p_0:=[-\frac{1}{3}(x^2+y^2)^2: \frac{2}{27}(x^2+y^2)^3]$ in $\P_s$ (in this case it is non-normal). 
\end{enumerate}
\end{lem}
\begin{proof} By the numerical criterion, a point $f=[f_4:f_6]$ is unstable if and only if there is  a point $u\in \P^1(x, y)$ such that $f_4$ and $f_6$ has multiplicity bigger than two and three at $u$ respectively. Without loss of generality, we may assume $u=[1:0]$, so that $y^3$ divides $f_4$ and $y^4$ divides $f_6$. Then it is easy to see that the corresponding sextic has  a triple point at $u$, with unibranch (i.e. a unique 
tangent line). So the surface $X_f$ has an $E_k$ or worse singularity. Conversely if $X_f$ has a  singularity of type $E_k$ or worse, then by multiplying by an element in $SL(2;\C)$ we may assume the singularity is of the form $[1:0:z_0]\in \P(1,1,2)$. In the affine chart where $x\neq0$,  the sextic is of the form $z^3+z f_4(1,y)+f_6(1, y)$.  It is easy to see that the only triple point must have $y=z=0$. Then it follows that $[f_4: f_6]$ is unstable. Similarly, it is easy to see that $X_f$ is stable if and only if it contains at worst $A_k$ singularities, i.e.  the sextic contains at worst double points.  If $X_f$ is polystable, then $[f_4: f_6]$ must be in the $SL(2;\C)$ orbit of $[ax^2y^2: bx^3y^3]$ for some non zero $[a:b]\in \P(2, 3)$. It is not hard to see that for $[a:b]\in \P(2, 3)$ not equal to $[-1/3: 2/27]$,  $X_f$ has exactly two $D_4$ singularities.
\end{proof}

\begin{rmk}
We remark that, in the context of rational elliptic surfaces (which is the
blow up of the base point of a complete anti-canonical system of degree $1$ Del Pezzo surface), Miranda \cite{Mir}
also analyzed the equivalent GIT stability and constructed 
the corresponding compactified moduli variety which is isomorphic to our $M_1'$. 
\end{rmk}

\subsubsection{Second step: Blow up}\label{dP1.blup.ss}

  For the compatibility with later discussions, 
we replace characters $x,y,z,w$ by $x',y',z',w'$ for the homogeneous coordinates for 
$\P(1,1,2,3)$. Recall that in the statement of Theorem \ref{double cover classification} when the Gromov-Hausdorff limit is the double cover of $\P(1,1,2)$, the branch locus could pass the vertex. 
This corresponds to the $z'^3$ term vanishing in the statement$F(x',y',z')$. 
 By Lemma \ref{quotient singularity}, it is easy to see that if we want the surface to have only quotient singularities, there must be a term of the form $z'^2f_2(x', y')$ where $f_2$ must have rank at least one.  On the other hand, for these surfaces,   there is no obvious reason that  they do not appear as the Gromov-Hausdorff limit of K\"ahler-Einstein surfaces. 
 Indeed we have explicit examples of such surfaces which admit K\"ahler-Einstein metric. The first is a one dimensional family of degree one K\"ahler-Einstein log Del Pezzo surfaces which are Gorenstein except one whose $f_2$ is rank two. 
 
 \begin{exa}
 We consider  a $\Z/2\Z$ action on the family of degree two surfaces $X_2^\lambda$ as studied in the end of Section \ref{degree two case}. The action is given by $[x:y:z:w]\mapsto [x:y:-z:-w]$. The fixed points are exactly the singularities of $X_2^\lambda$.   One can check that  for $\lambda \neq [1:0], [0:1]$, the quotient $X_1^\lambda$ is a degree one log Del Pezzo surface with exactly two $D_4$ singularities. It is interesting that these surfaces admit a $\C^*$ action and correspond exactly to the polystable points in $M_1'$, except $p_0$.  From the discussion in the end of Section \ref{degree two case} we see they all admit K\"ahler-Einstein metrics. 
 
 For the surface $X_2^\infty$ the action fixes also the $A_1$ singularity $[0:0:1:0]$, so the quotient $X_1^\infty$ has two $D_4$ singularities and one $\frac{1}{4}(1,1)$ singularity. Denote the embedding $\P(1,1,2)\hookrightarrow \P^3$ by $[x':y':z']\mapsto [z': x'^2: x'y': y'^2]$. Then the bi-anti-canonical map realizes $X_1^\infty$ as a double cover of $\P(1,1,2)\subset \P^3$ branched along the curve isomorphic to $z'^2x'y'+x'^3y'^3$. Indeed, $|-2K_{X_1^{\infty}}|=|-2K_{X_2^{\infty}}|^{\mathbb{Z}/2\mathbb{Z}}=|\mathcal{O}_{\mathbb{P}(1,1,1,2)}(2)^{\mathbb{Z}/2\mathbb{Z}}|$ which is spanned by $x^2, xy, y^2, z^2$ 
so  the branch locus is $xyz(z-xy)$.  The latter is isomorphic to the sextic described above. 

So $X_1^\infty$ corresponds to the case that $f_2$ has rank two.  Clearly $X_1^\infty$ admits a K\"ahler-Einstein metric, as a global quotient of $\P^1\times \P^1$. 
 \end{exa}
 
 The next example, which will be important in our further modification, is a degree one K\"ahler-Einstein log Del Pezzo surface which corresponds to $f_2$ being rank one. 
 
 \begin{exa}\label{d1new} Consider the degree two surface $X_2^{\gamma_0}$ with $\gamma_0=[1:-1]$. 
 It has two $A_3$ singularities, one at $[1:0:0:0]$ and one at $[0:1:0:0]$.  Now consider the involution  $\sigma: X_2^{\gamma_0}\rightarrow X_2^{\gamma_0}$ which sends $[x:y:z:w]$ to $[x:-y:-z:-w]$. Then $\sigma$ has two fixed points exactly at the two singularities. It is straightforward to check that the quotient, which we will denote by $X_1^e$ from now on, has one $A_7$ singularity and one $\frac{1}{8}(1,3)$ singularity. 
$|-2K_{X_1^e}|$ is determined by the sections $\{x^2, y^2, yz, z^2\}\in H^0(\P(1,1,1,2), \O(2))$, and this defines a double covering map from $X_1^e$ to the quadric cone in $\P^3$. 
 The corresponding involution $\sigma$ maps $[x:y:z:w]$ to $[-x:-y:-z:-w]=[-x: y:z:w]$ (the identity holds on $X_1^e$). Then the fixed locus of $\sigma$ consists of the curve $w=0$ and the curve $x=0$. Denote again the embedding $\P(1,1,2)\hookrightarrow \P^3$ by $[x':y':z']\mapsto [z': x'^2: x'y': y'^2]$. The branch locus in $\P(1,1,2)$ is isomorphic to the sextic $z'^2x'^2-z'y'^4=0$. So $X_1^e$ corresponds to that $f_2$ has rank one.  Again $X_1^e$ admits a K\"ahler-Einstein metric by the discussion in the end of Section \ref{degree two case}. 
 \end{exa}

 We have a refined classification than Corollary \ref{equation degree one}. 
 
 \begin{lem} \label{Z8 is unique}
Let $X_\infty$ be the Gromov-Hausdorff limit of a sequence of degree one K\"ahler-Einstein Del Pezzo surfaces. If it is a hypersurface in  $\P(1,1,2, 3)$ of the form   $w^2=F_6(x, y, z)$, then either $F_6$ has a term $z^3$, or $F_6$ is equivalent to $z^2(x^2+y^2)+zg_4(x, y)+g_6(x, y)$ or $X_0$ is isomorphic to $X_1^e$.  
\end{lem}

\begin{proof}Consider the case when $F_6$ contains no $z^3$ term. Then we claim the term $z^2f_2(x,y)$ must not vanish. Otherwise $F_6=zf_4(x,y)+f_6(x,y)$.  Then in the affine chart $\{z\neq 0\}$ in $\P(1,1,2,3)$ we have equation $w^2=f_4(x,y)+f_6(x,y)$ then by the Lemma  \ref{quotient singularity},  $X_\infty$ has a non quotient singularity, so it can not be a Gromov-Hausdorff limit by Theorem \ref{orbifold compactness}. So up to equivalence we may assume the $z^2$ term in $F_6$ is of the form $z^2(x^2+y^2)$ or $z^2x^2$. In the former case we are done, so we assume the latter. Then we can write 
$$F_6(x,y, z)=z^2x^2+a zy^4+bzxf_3(x, y)+f_6(x,y). $$
Now if $a=0$, then again in the affine chart $\{z\neq 0\}$ we have equation $w^2=x^2+bxf_3(x,y)+f_6(x, y)$. Then by a change of variable at $(0, 0,0)$ we may assume it is locally equivalent to $w^2=x^2+a_1 xy^3+a_2 xy^5+a_3 y^6$. It is easy to see this is either non-normal or has a $A_i$ singularity $i\ge 5$ at the origin. The corresponding singularity on $X_0$ is a $(\Z/2\Z)$-quotient by the action $(x, y, w)\mapsto (-x, -y, -w)$. So $X_0$ is either non-normal or has an orbifold point of order at least 12, thus it can not admit a K\"ahler-Einstein metric by Theorem \ref{Bishop-Gromov}. 

So $a\neq 0$, then by a change of variables $y\mapsto y+cx$ and $z\mapsto z+g_2(x,y)$ we may assume 
$$F_6(x, y, z)=z^2x^2+zy^4+f_6(x,y). \ \ (*)$$
 $X_1^e$ is isomorphic to the surface defined by $w^2=z^2x^2+zy^4$. The one  parameter subgroup $\lambda(t)=(t^2, t, 1, t^2)$ degenerates every surface defined by $(*)$ to $X_1^e$ as $t$ tends to zero. Since $X_1^e$ admits a K\"ahler-Einstein metric, it has vanishing Futaki invariant. By Theorem \ref{KEtoKstability} we see $X_\infty$ must be isomorphic to $X_1^e$. 

\end{proof}

We first construct  a moduli space for surfaces with $f_2$ being rank two, and we will show these surfaces are parametrized exactly by a weighted blow up of  $M_1'$ at $p_0$. The surfaces are defined by
 
\begin{equation}\label{dP1.exc}
w'^2=z'^2(x'^2+y'^2)+z'g_4(x', y')+g_6(x',y'). 
\end{equation}

Similarly as before, by considering the translation $z' \mapsto z'+a_2(x',y')$ 
for certain quadric $a_2(x',y')$, we may assume $g_4$ lies in the space  $T(x', y'):=\C(x'+iy')^4\oplus \C(x'-iy')^4$, which is the $SO(2;\C)(\cong\C^*)$-invariant complement  to the linear subspace of ${\it Sym}^4(\C x'\oplus \C y')$ consists of those divisible by $(x'^2+y'^2)$. 
In this way, we can obtain GIT quotient $\P_e^{ss}//SO(2;\C):=\mathbb{P}(1,1,2,2,2,2,2,2,2)^{ss}//SO(2;\C)$ 
which parametrizes surfaces of the form (\ref{dP1.exc}). Here we need to specify the weight of $SO(2;\C)\cong\C^*$ on the linearization, and we choose the natural one, so the action corresponding to $(x'+iy')\mapsto \mu (x'+iy')$, $(x'-iy')\mapsto \mu^{-1}(x'-iy')$ has weight 
\begin{equation}\label{wt}
(4, -4, 6, 4, 2, 0, -2, -4, -6), 
\end{equation}
with respect to the basis consists of 
\begin{equation*}\label{bas1}
(x'+iy')^4, (x'-iy')^4, 
\end{equation*}
and 
\begin{eqnarray*}\label{bas2}
&&(x'+iy')^6, (x'+iy')^5(x'-iy'), \\&&
(x'+iy')^4(x'-iy')^2, (x'+iy')^3(x'-iy')^3, \\&&
(x'+iy')^2(x'-iy')^4, (x'+iy')(x'-iy')^5, (x'-iy')^6. 
\end{eqnarray*}

Then we have the following. 
\begin{lem}
The GIT quotient  $\P_e^{ss}//SO(2;\C)$ with respect to the action with weight (\ref{wt}) above parametrizes log Del Pezzo surfaces, i.e. a polystable sextic defined by $[g_4:g_6]\in P_e$ has only quotient singularities, or more precisely, the corresponding Del Pezzo surface has exactly one $\frac{1}{4}(1,1)$ singularity besides canonical singularities. 
\end{lem}
\begin{proof}
It is easy to check that if a sextic has the form $z'^2(x'^2+y'^2)+z'(a (x'+iy')^4+b(x'-iy')^4)+g_6(x', y')$ with $a, b\neq 0$, then it has only double points away from the vertex. If $a=b=0$, then for it to be stable, it has at most double points, and for it to be polystable, it has exactly two $D_4$ singularities besides the vertex. If $a\neq0$ and $b=0$, then, if it is stable, the sextic has at most double points, and if it is semistable, then it degenerates to $z'^2(x'^2+y'^2)+a (x'+iy')^3(x'-iy')^3$, which has two $D_4$ singularities. 
\end{proof}

When we prove the moduli space we constructed in the end has property (KE) we need to show:

\begin{lem} \label{CM exceptional}
A surface of the form (\ref{dP1.exc}) that admits a K\"ahler-Einstein metric must be GIT polystable with respect to the chosen linearization as above. 
\end{lem}

\begin{proof} 
 This does not follow directly from the general Theorem \ref{CM stability}, as the group $SO(2;\C)\cong\C^*$ has non trivial characters. But in our case this can be done by explicit analysis as follows. 
Notice that since $\P_e$ contains a point parametrizing a K-polystable log Del Pezzo surface (e.g. $X_1^\infty$), the CM line bundle must be isomorphic to $\O(k)$ for $k>0$. This follows from the proof of Theorem \ref{CM stability}. $X_1^\infty$ corresponds to the vector $v=[0:0:0:0:0:1:0:0:0]$ in $\P_e$ with respect to the quasi-homogeneous 
coordinates as above. So the weight of the action on the CM line bundle must also be the natural one as above, for otherwise it is easy to see that $v$ is unstable. 
\end{proof}

The second step toward the construction of $M_1$ is to replace the point 
$[p_0]\in M_1'$ (which corresponds to a non-normal surface) by the above GIT quotient. 

\begin{thm}
There is a blow up $M_1''\rightarrow M_1'$ at $[p_0]$ (with a non-reduced ideal) so that $M_1''$ is an analytic moduli space for degree one log Del Pezzo surfaces. The exceptional divisor $E$ is isomorphic to $\P_e^{ss}//SO(2;\C)$. Moreover, a point $s\in M_1''$ parametrizes the polystable sextic hypersurface $X_s$ defined by it, and $s\in E$ if and only if  the sextic passes through the vertex $[0:0:1]$.
\end{thm}

\begin{proof}
Let $\tilde{\mathbb{A}}\simeq {\it Sym}^4(\C x\oplus \C y)\oplus {\it Sym}^6(\C x
\oplus \C y)$ be the cone over $\mathbb{P}_s$. In the tangent space at the point $p_0=(-\frac{1}{3}(x^2+y^2)^2, \frac{2}{27}(x^2+y^2)^3)$, we  take an $SO(2;\C)$-invariant Luna \'etale 
slice $\mathbb{A}_f:=p_0+\{T(x, y)\oplus {\it Sym}^6(\C x\oplus \C y)\}$ in $\tilde{\mathbb{A}}$.
To include surfaces of the form (\ref{dP1.exc}), let $\A_g=T(x', y')\oplus {\it Sym}^6(\C x'\oplus \C y')$, and  we consider the family of surfaces over $\A_g\times \C^*$ where we associate $(g_4, g_6, t)$ the sextic 
\begin{equation} \label{sextic}
tz'^3+z'^2(x'^2+y'^2)+z'g_4(x',y')+g_6(x',y').
\end{equation}
Making the change of variable 
$$x':=tx, y':=ty, z':=z-\frac{t}{3}(x^2+y^2), $$
and 
$$f_4(x, y)=-\frac{t^2}{3}(x^2+y^2)^2+ t^3 g_4(x,y); $$
$$f_6(x, y)=\frac{2t^3}{27}(x^2+y^2)^3-\frac{t^4}{3}(x^2+y^2)g_4(x,y)+t^5g_6(x,y), $$
the sextic in equation (\ref{sextic}) is then transformed into  the form
$$t[z^3+f_4(x, y)z+f_6(x, y)]. $$
Hence it corresponds to the point $[f_4(x, y): f_6(x,y)]\in \A_f\subseteq \P_s$. If we keep $g_4$ and $g_6$ fixed, and let $t$ tend to zero this converges exactly to the point $p_0$.

The equation (\ref{sextic}) defines a family of sextics over the trivial $\P_{x', y', z'}(1,1,2)$ bundle $\mathcal{P}'$ over $\A_g\times \C^*$, and it extends obviously over $\A_g\times\C$, which is the cone over the blow up $\B_g$ of $\A_g$ at $0$.  This family is invariant under $\C^*$ action $\lambda. (t, g_4, g_6):=(\lambda^{-1}t, \lambda g_4,  \lambda^2 g_6)$, and thus 
descended to a family over $\B_g$.
 The above change of variables indeed defines an isomorphism $\Psi$ between $\mathcal{P}=\P_{x,y,z}(1,1,2)\times (\A_f\times\C^*)$, and induces a $\C^*$ action on $\A_f$. 
 We decompose $\A_f$ as $\A_f=p_0+(L_1\oplus L_2)$, where 
 $$L_1:=\{(f_4(x,y), -\frac{1}{3} (x^2+y^2)f_4(x,y))\}\mid f_4\in T_{(x,y)}\}, $$and 
$$L_2:= {\it Sym}^6(\C x\oplus \C y)\subset \A_f. $$ 
Denote the associated ideals of $L_i+p_0$ in $\mathbb{A}_f$ by 
$I_{(L_i+p_0)}$. Then we define $\B_f$ to be the  blow up of $\mathbb A_f$ at $I_{(L_1+p_0)}^2+I_{(L_2+p_0)}$. The exceptional divisor is isomorphic to $\P_e$.
Then by pulling back by $\Psi$ we obtain a flat family of sextics over $\B_f$, and the exceptional divisor parametrizes sextics of the form (\ref{dP1.exc}).

Similarly to the degree $2$ case, we consider GIT of  $\B_f$ 
with respect to the $SO(2;\C)$-action, and get 
a certain blow up $\B_f^{ss}//SO(2;\C)\rightarrow \A_f//SO(2;\C)$. 
This induces a blow up of 
$\mathbb{P}_s//SL(2;\C)$ at $[p_0]$, with exceptional divisor $E\cong \P_e^{ss}//SO(2;\C)$. We denote this by $M_1''\rightarrow M_1'$. 

From the construction, as in the previous section, $M_1''$ is an analytic moduli space and a coarse moduli of an algebraic stack which is constructed by gluing 
$$[\B_f^{ss}/SO(2;\C)]$$ naturally with 
$$[(\mathbb{P}_s^{ss}\setminus (PGL(2;\C).p_0))/PGL(2;\C)]$$ in our context. 
\end{proof}

\subsubsection{Construction of moduli: further modifications}

We have a further refinement of Corollary \ref{equation degree one}, parallel to Lemma \ref{Z8 is unique}.

\begin{lem} \label{toric is unique}
Let $X_\infty$ be the Gromov-Hausdorff limit  of a sequence of degree one K\"ahler-Einstein Del Pezzo surfaces. Then $X_\infty$ is a sextic hypersurface in $\P(1,1, 2, 3)$ of the form  $x_4^2=f_6(x_1, x_2, x_3)$, or isomorphic to the toric surface $X_1^T$. 
\end{lem}
\begin{proof} By Theorem \ref{double cover classification}, we may assume $X_\infty$ is a degree 18 hypersurface in $\P(1,2, 9,9)$ of the form $x_4^2=f_{18}(x_1, x_2, x_3)$ not  passing through the point $[0:0:1]$. So we may assume $f_{18}(x_1, x_2, x_3)= x_3^2+g_{18}(x_1, x_2)$. If  the term $x_2^9$ appears in $g_{18}$, then the one parameter subgroup $\Lambda$ acting with weight $(0, 9, 2, 2)$  degenerates $x_4^2-f_{18}$ to $x_4^2-x_3^2-a x_2^9$. This induces a test configuration for $X_\infty$ with central fiber isomorphic to $X_1^T$. Since $X_1^T$ has vanishing Futaki invariant,  and $X_\infty$ is K-polystable, we conclude that $X_\infty$ must be isomorphic to $X_1^T$.  If $x_2^9$ does not appear in $g_{18}$, then  the one parameter subgroup $\Lambda$ acting with weight $(0, 0, 1,1)$ degenerates $x_4^2-f_{18}(x_1,x_2, x_3)$ to $x_4^2-x_3^2$. Again this induces a test configuration for $X_\infty$ with central fiber the nonnormal hypersurface $Y$ defined by $x_4^2-x_3^2=0$. We claim this has zero Futaki 
invariant, thus contradicting the fact that $X_\infty$ is K-polystable. To see the claim, note that the Futaki invariant for a $\C^*$-action on a connected fixed component in the Hilbert scheme is constant. Since $X_1^T$ obviously degenerates to $Y$ and is fixed by the same $\Lambda$, we can compute the Futaki invariant on $X_1^T$, which is zero since it  is K\"ahler-Einstein. 
\end{proof}

The analytic moduli space $M_1''$ constructed in the previous section does not have property (KE),  since it does not parametrize the two examples $X_1^e$ and $X_1^T$ which we are unable to show that they can not appear as a Gromov-Hausdorff limit.   So we have to make a modification of $M_1''$. Now the only problem is to fit these two into $M_1''$. We first illustrate the phenomenon of modification of GIT by  a simple example. 

\begin{exa}
Let $\C^*$ act linearly on $\C^2$ by $t. (z_1, z_2)=(t z_1, z_2)$. Then the quotient is isomorphic to $\C$, and the polystable locus are points on the line $\{0\}\times\C$. If we remove the origin $(0,0)$, then the quotient is again isomorphic to $\C$, but the polystable locus differs from the previous one in that the orbit of the origin is replaced by the punctured line $\C^*\times\{0\}$. 
\end{exa} 

Our situation is very similar to this. We first investigate the $\Q$-Gorenstein deformation of $X_1^T$ studied in Section 3. Adopting the notation there, we have:

\begin{lem}
A point $v=(v_1, v_2, v_3)\in \Def(X_1^T)$ is polystable under the action of $\Aut^0(X_1^T)$ if and only if $v_1$, $v_2$ and $v_3$ are all non-zero or all zero, and $(0,0,0)$ is the only strictly polystable point. 
\end{lem}
\begin{proof} If $v_1=0$, then we can destabilize $v$ by the one-parameter subgroup $\lambda(t)=(t^{-1}, 1)$. If $v_2=0$, then we can destabilize $v$ by the one-parameter subgroup $(1, t^{-1})$. If $v_3=0$, then we can destabilize $v$ by the one-parameter subgroup $(t^3, t^2)$. If all the $v_i$'s are non-zero, then for $\lambda(t)=(t^a, t^b)$  to destabilize $v$ we need $a-b\geq0$, $-3a+6b\geq0$, and $-3a-3b\geq0$. It is easy to see that no non-trivial such pair $(a, b)$ exists.   
\end{proof}

To fill $X_1^{T}$ in our moduli, since we may locally identify $\Kur(X_1^T)$ with $\Def(X_1^T)$, and 
the $(\mathbb{C}^*)^2$-action on $\Kur(X_1^T)$ is compatible with the one on $\Def(X_1^{T})$, it suffices to study the GIT on $\Def(X_1^T)$. 
By the above lemma, the stable points all represent canonical log Del Pezzo surfaces with at most a unique $A_k (k\leq 7)$ singularity, and the polystable point $0$ represents $X_1^T$. The GIT quotient $Q$ is then smooth at $[X_1^{T}]$. 
The semistable orbit  $(0,v_2,v_3)$ (where $0<|v_2|^2+ |v_3|^2\ll 1$)  represent a log Del Pezzo surface with a unique $A_8$ singularity. 
Since it is unique up to isomorphism by \cite{Furu}, we denote it by $X_1^a$. Due to 
Lemma \ref{degree one stable classification}, it has discrete automorphism group and it is parametrized by a point $u_0$ in $M_1''\setminus E$. 

Consider the analytic subset  $\Kur'(X_1^T)$ of $\Kur(X_1^T)$ which represents only canonical log Del Pezzo surfaces, i.e. that consists of points with $v_2\neq 0$ and $v_3\neq 0$. Then the corresponding quotient $Q'$ can be identified with the previous quotient $Q$, which identifies every stable orbit, except the orbit of $X_1^a$ is replaced by $X_1^T$. $Q'$ can be viewed as the universal deformation space $X_1^a$. There is an analytic neighborhood $U$ of $u_0$,  and an embedding $\iota: U\rightarrow Q'=Q$ such that $\iota(u_0)=0$, and  $u$ and $\iota(u)$ parametrize equivalent surfaces. In terms of stack language,  the open embedding of stacks 
$[(\Kur(X_1^T)\setminus \Kur(X_1^T)')/(\mathbb{C}^*)^2]\hookrightarrow [\Kur(X_1^T)/(\mathbb{C}^*)^2]$ 
induces an isomorphism of the categorical moduli. 
Now we can simply define $M_1'''=M_1''$ as a variety and only change the surface parametrized by $u_0$ from $X_1^a$ to $X_1^T$. Then it is clear that $M_1'''$ is again an analytic moduli space of degree one log Del Pezzo surfaces.  So this modification takes care of the point $X_1^T$. 

Now we treat $X_1^e$ in a similar fashion. First notice that the linear system $|-2K_{X_1^e}|$ realizes $X_1^e$ as the double cover of $\P(1,1,2)$, thus $\Aut^0(X_1^e)$ is induced from  $\Aut(\P(1,1,2))$. Then one sees that $\Aut^0(X_1^e)\cong\C^*$ corresponds to the scaling $\lambda(t)=(t^2, t, 1, t^2)$. By Lemma \ref{local-global}, we have 
$$\Def(X_1^e)=\Def'\oplus \Def_1\oplus \Def_2,  $$
where $\Def'$ corresponds to equisingular deformations, $\Def_1$ corresponds to deformations of the local singularity at $[0 :0:1:0]$,  and  $\Def_2$ corresponds to deformations of the local singularity at $[1:0:0:0]$. By applying again the Main Theorem of \cite{Ma}, it follows that 
$\Def_1$ is two dimensional and $\Def_2$ is seven dimensional. Thus by dimension counting we must have $\Def'=0$. 
We can write down a semi-universal deformation family:  
\begin{equation}\label{x1e.def} 
 w^2=z^2x^2+zy^4+a_1 z^3+a_2z^2y^2+\sum_{i=0}^6 b_ix^iy^{6-j},  
\end{equation}
In particular, note that we have $\Aut(X_1^{e})$-invariant affine 
 versal deformation space $\Kur(X_1^e)$ as claimed in the explanation after Lemma \ref{local-global}  and in this case, $\Kur(X_1^{e})$ can be identified globally with the tangent space 
 $\Def{X}_1^e$ so that $(a_1, a_2)\in \Def_1$ and $(b_0, \cdots, b_6)\in \Def_2$. 

It is also easy to see the weights of the action is
 $$\lambda(t). (a, b)=(t^{-4}, t^{-2}, t^{8}, t^6, \cdots, t^2). $$
So in the local GIT quotient by $\Aut(X_1^e)$ a point $(a,b)$ is stable if and only if $a\neq 0$ and $b\neq0$ in which case $X_{a, b}$ has either a unique $A_k(k\leq 6)$ singularity or a $\frac{1}{4}(1,1)$ plus $A_k(k\leq 6)$ singularity. 

When we remove the subspace $\{0\}\oplus \Def_2$, every point becomes stable. In particular, the quotient of the subspace $(a, 0)$ with $a\neq 0$ is exactly a $\P^1$, which parametrizes surfaces in $M_1''$ 
$$w^2=a_1z^3+z^2x^2+zy^4+a_2 z^2y^2, $$
and intersects the exceptional divisor at one point corresponding to $a_1=0$.  
It is easy to see that $\lambda(t)$ degenerates all these surfaces to $X_1^e$ as $t$ tends to infinity, so they could not admit K\"ahler-Einstein metrics, and we need to remove them. 
Notice this family does not include the point corresponding to $X_1^a$, so we can make a further modification simultaneously as the previous one.  When we add the the subspace $\{0\}\oplus \Def_2$, the point $(a, 0)$ with $a\neq 0$ become semistable and in the GIT quotient this is contracted to the point $0$. To be more precise we take the neighborhood $U$ in $\Def(X_1^e)$ consisting of points $(a, b)$ with $||a|-1|\ll 1$ and $|b|\ll 1$, and the quotient $V$ by $\C^*$ gives rise to a tubular neighborhood of the $\P^1$ in $M_1''$. When we add the subspace $\{0\}\oplus \Def_2$ we have that $V$ gets mapped to a neighborhood of $0$ in the local GIT, with $\P^1$ contracted to $0$. 

As before the GIT on $\Kur(X_1^e)$ and on $\Def(X_1^e)$ are equivalent so this allows us to perform the contraction in an analytic neighborhood of the $\P^1$ inside $M_1'''$. We obtain a new analytic moduli space $M_1$, which enjoys the Moishezon property. 
Thus it has a natural structure of an algebraic space as well.  

Theorem \ref{MT} in degree one case then follows from the theorem below. 

\begin{thm} \label{degree one perfect}
$M_1$ has property (KE). 
\end{thm}
\begin{proof}
The proof is very similar to Theorem \ref{degree two perfect}. By Lemma \ref{toric is unique} we only need to show that if a $X\in M_1^{GH}$ is a sextic hypersurface  in $\P(1,1,2,3)$ defined by $w^2=f_6(x, y,z)$, then it is parametrized by some element in $M_1''$. If $f_6$ contains a term $az^3$ with $a\neq0$, then it is parametrized a point $u$ by $\P_s$. Then by Theorem \ref{KEtoKstability} and Theorem \ref{CM stability},  keeping in mind that $\P_s$ has Picard rank one, we conclude that $u$ is polystable under the $SL(2;\C)$ action, thus $X$ is parametrized by a point $p$ in $M_1'$. Then $X$ can not be isomorphic to $X_1^T$ or the $\P^1$ family above. So $X$ is parametrized by a point in $M_1$.  If the term $z^3$ does not appear in $f_6$, then by Lemma \ref{Z8 is unique} and Lemma \ref{CM exceptional}  $X$ is either isomorphic to $X_1^e$ or  is parametrized by a polystable point $u\in \P_e$. 
Again this point $u$ can not be on the $\P^1$ and this means that $u$ is in $M_1$. 

\end{proof}

We can construct a KE moduli stack $\mathcal{M}_1$ 
by gluing the previously constructed moduli stack with $[U/\Aut(X_1^e)]$ where $U$ is some open $\Aut(X_1^e)$-invariant neighborhood of $0\in \Kur(X_1^e)$ 
(along $[(U\setminus (\{0\}\oplus \Def_2))/\Aut(X_1^e)]$). 
Recall that in this case, we identified globally $\Kur(X_1^e)$ and 
$\Def(X_1^e)$. 
We can show that with a small enough $\Aut(X_1^e)$-invariant open neighborhood $U$ of $0$ in $\Kur(X_1^e)$, a stack  
$[(U\setminus (\{0\}\oplus \Def_2)))/\Aut(X_1^e)]$ 
has a natural \'etale morphism to the previously constructed 
moduli stack so that the glueing is possible. Indeed, the $\mathbb{Q}$-Gorestein deforming component (cf. \cite[section 5]{KS}) of a Luna \'etale slice in 
the Hilbert scheme $\Hilb(\mathbb{P}(H^{0}(X_1^T,-K_{X_1^T}^{\otimes m})))$ at $[X_1^T]$ with respect to the standard $\SL$ action 
is \'etale locally semi-universal deformation by the universality of the Hilbert scheme. Then the \'etale local uniqueness of 
semi-universal family tells us it is actually \'etale locally equivalent with $U$ including the family on it. Then the assertion follows from the universality of Hilbert scheme again. Note that especially $U$ includes the subspace $\Def_{1}\oplus \{0\}$ so that the categorical moduli of 
the open immersion $[(U\setminus (\{0\}\oplus \Def_2)))/\Aut(X_1^e)] \hookrightarrow [U/\Aut(X_1^e)]$ represents the contraction 
of $\mathbb{P}^1$. 

Then $\M_1$ is a KE moduli stack and $M_1$ constructed above is KE moduli space. This completes the proof of Theorem \ref{MT} for degree $1$ case as well. Note that our contraction of $\mathbb{P}^1$ on the coarse quotient is constructed just on an \'etale cover, not 
a priori an open substack. Indeed it is not, although we omit the lengthy proof for that. 
This is the reason our argument is not enough to show $M_1$ is a (projective) variety. 
Completely as before, there is a natural anti-holomophic involution on $M_1$ which gives rise to the complex conjugation.

\subsubsection{A remark on a conjecture of Corti}

In the paper \cite{Cor}, Corti conjectured the following, 
motivated by the possibility of using birational geometry to get 
certain ``nice" integral models over a discrete valuation ring:

\begin{conj}[{\cite[Conjecture 1.16]{Cor}}]
For an arbitrary smooth punctured curve $C\setminus \{p\}$ and a smooth family of 
Del Pezzo surfaces $f\colon \mathcal{X}\to (C\setminus \{p\})$ over it, 
we can complete it to a flat family $\bar{f}\colon \bar{\mathcal{X}}\to C$ which satisfies: 
\begin{itemize}
\item $\mathcal{X}$ is terminal. 
\item $\mathbb{Q}$-Gorenstein index of $\bar{\mathcal{X}}_{p}$ is either $1, 2, 3$ or $6$ and $-6K_{\bar{\mathcal{X}}_{p}}$ is very ample. 
\end{itemize}
\end{conj}

\noindent 
He called $\bar{\mathcal{X}}$ the \textit{standard model}.  We have the following partial   solution to the above;
it is rather weak, in the sense we permit base change, but 
on the other hand we even have a classification of the possible central fiber. 

\begin{prop}
For an arbitrary smooth punctured curve $C\setminus \{p\}$ and a smooth family of 
Del Pezzo surfaces $f\colon \mathcal{X}\to (C\setminus \{p\})$ over it, 
possibly ramified base change $p' \in \tilde{C}\rightarrow C$ (with $p'\mapsto p$), we can fill the punctured family $\mathcal{X}\times _{(C\setminus \{p\})} (\tilde{C}\setminus \{p'\})$ 
to a flat family $\bar{\mathcal{X}}' \rightarrow \tilde{C}$ such that: 
\begin{itemize}
\item $\bar{\mathcal{X}}'$ is terminal. 
\item $\mathbb{Q}$-Gorenstein index of $\bar{\mathcal{X}}'_{p'}$ is either $1, 2$ and 
$-6K_{\bar{\mathcal{X}}'_{p'}}$ is very ample. 
\end{itemize}
\end{prop}

\begin{proof}
We have constructed the moduli stack $\M_1''$ by gluing quotient stacks of certain GIT semistable locus 
(subsection \ref{dP1.blup.ss}). 
From the construction, it is universally closed stack and 
it parametrizes log del Pezzo surfaces of $\mathbb{Q}$-Gorenstein index 
$1$, $2$. The $\mathbb{Q}$-Gorenstein property of $\mathcal{X}$ follows from 
our construction as well. 
\end{proof} 

\subsubsection{Relation with moduli of curves}\label{curve.1}

We expect the KE moduli variety $M_1$ to be  a divisor of 
one of the geometric compactifactions of moduli of curves with genus $4$. 
Especially we suspect  that our moduli $M_1$ is
the prime divisor of $\overline{M}_4(a)$ with $\frac{23}{44}<a<\frac{5}{9}$ in \cite{CJL}. Note that it is the moduli of Hilbert polystable 
canonical curves. 

\section{Further discussion} \label{K moduli}
\subsection{Some remarks}
\subsubsection{Lower bound of the Bergman function}
The main technical part in the proof of Proposition \ref{Fano limit} is a uniform lower bound of the Bergman function. Let $(X, J, \omega, L)$ be a polarized K\"ahler manifold, then for any $k$ there is an induced metric on $H^0(X, L^k)$. The Bergman function is defined by 
$$\rho_{k, X}(x)=\sum |s_\alpha|^2(x), $$
where  $\{s_\alpha\}$ is any orthonormal basis of $H^0(X, L^k)$. Kodaira embedding theorem says that for fixed $X$, and for sufficiently large $k$ the Bergman function is always positive.   It is proved in \cite{DS} that for a $n$ dimensional K\"ahler-Einstein Fano manifold $(X, J, \omega)$, we always have $\rho_{k, X}(x)\geq \epsilon$ for some integer $k$ (and thus every positive multiple of $k$) and $\epsilon>0$ depending only on $n$. This was named ``partial $C^0$ estimate" in \cite{Tian1} and it is also proved there for two dimensional case. It was explained in \cite{DS} that one may not take $k$ to be all sufficiently large integers, and in our proof of the main theorem we have seen examples, see Remark \ref{Tian conjecture}.  Indeed we  found explicitly all the integers $k$ that we need to take in each degree in order to ensure a uniform positivity of Bergman function for all K\"ahler-Einstein Del Pezzo surfaces. (Compare the strong partial $C^0$ estimate in \cite{Tian1}, Theorem 2.2):
\begin{itemize}
 \item $d=4, 3$:  $k\geq 1$;
 \item $d=2$:  $k=2l$, with $l\geq 1$;
 \item $d=1$:  $k=6l$, with $l\geq 1$.
\end{itemize}

\subsubsection{K\"ahler-Einstein metrics on del Pezzo orbifolds}

As a consequence of our main Theorem \ref{MT}, we have a complete classification of K\"ahler-Einstein Del Pezzo surfaces with at worst canonical singularities in terms of K-polystability. 
\begin{cor}
 Let $X$ be a Del Pezzo surface with at worst canonical singularities. Then
$$X \mbox{ admits a K\"ahler-Einstein metric} \Longleftrightarrow X \mbox{ is K-polystable}.$$  
\end{cor}
\begin{proof}
 The direction ``$\Longrightarrow$'' is known by Theorem  \ref{KEtoKstability}. To prove the other direction, suppose that $X$ is K-polystable and with at worst canonical singularities (in particular it is automatically $\Q$-Gorenstein smoothable). Then by 
Theorem \ref{CM stability} $X$ is also polystable with respect to the stability notions that we used in the construction of our moduli spaces, i.e. $[X] \in M_d$. Thus $X$ admits a K\"ahler-Einstein metric as a consequence of  Theorem \ref{MT}.
\end{proof}

The above result gives the answer to the conjecture of Cheltsov and Kosta (\cite{CK}, Conjecture 1.19) on the existence of K\"ahler-Einstein metrics on canonical Del Pezzo surfaces. In particular, we have the following exact list of possible singularities that can occur. Let $(X,\omega)$ be a degree $d\leq 4$ Del Pezzo surface with canonical singularities, then it admits a K\"ahler-Einstein metric if and only if $X$ is smooth or
\begin{itemize}
 \item $d=4$: Sing($X$) consists of only two $A_1$ singularities and $X$ is simultaneously diagonalizable, or exactly four singularities (in which case $X$ is isomorphic to $X_4^T$);
 \item $d=3$: Sing($X$) consists of only points of type $A_1$, or of exactly three points of type $A_2$ (in which case $X$ is isomorphic to $X_3^T$);
 \item $d=2$: Sing($X$) consists of only points of type $A_1$, $A_2$, or of exactly two $A_3$ singularities;
 \item $d=1$: Sing($X$) consists of only points of type $A_k$ $(k\leq 7)$, or of exactly two $D_4$ singularities, and $X$ is not isomorphic to one the surfaces in the $\P^1$ family in the last section. 
\end{itemize}

As we have seen, the class of log Del Pezzo surfaces with canonical singularities is not sufficient to  construct a KE moduli variety. In particular we have found some $\Q$-Gorenstein smoothable K\"ahler-Einstein log Del Pezzo surfaces, hence K-polystable, with non-canonical singularities. Thus it is natural to ask the following differential geometric/algebro-geometric question: do there exist other $\Q$-Gorenstein smoothable K\"ahler-Einstein/K-polystable log Del Pezzo surfaces besides the ones which appear in our KE moduli varieties? 
If the answer to the above question is negative (as we conjecture) then the Yau-Tian-Donaldson conjecture for K-polystability also holds for the class of $\Q$-Gorenstein smoothable Del Pezzo surfaces. For this it is of course sufficient to prove the following: let $\pi\colon \mathcal{X} \to \Delta$ be a $\Q$-Gorenstein deformation of a K-polystable Del Pezzo surface $X_0$ over the disc $\Delta$ such that the generic fibers $X_t$ are smooth (hence admit  K\"ahler-Einstein metrics). Then $X_0$ admits a K\"ahler-Einstein metric $\omega_0$, and $(X_0, \omega_0)$ is the Gromov-Hausdorff limit of a sequence of K\"ahler-Einstein metrics on the fibers $(X_{t_i},\omega_{t_i})$ for some sequence $t_i\rightarrow 0$.

\subsection{On compact moduli spaces}

In this final section, we would like to formulate a conjecture about the existence of certain compact moduli spaces of K-polystable/K\"ahler-Einstein Fano varieties. Before stating our conjecture, we recall some important steps in the history of the construction and compactifications of moduli spaces of varieties.

For complex curves of genus $g\geq 2$, the construction of the moduli spaces, and their ``natural'' compactifications, was completed during the seventies by Deligne, Mumford, Gieseker and others \emph{using} GIT. The degenerate curves appearing in the compactification are the so-called ``stable curves'', i.e., curves with nodal singularities and discrete automorphisms group. Let us recall that these compact moduli spaces have also a ``differential geometric'' interpretation. It is classically well-known that every curve of genus $g$ has a unique metric of constant Gauss curvature with fixed volume. As the curves move towards the boundary of the Deligne-Mumford compactification, the diameters, with respect to the constant curvature metrics, go to infinity and finally these metric spaces  ``converge'' to a complete metric with constant curvature and hyperbolic cusps on the smooth part of a ``stable curve''.
 
The construction of compact moduli spaces of higher dimensional polarized varieties turns out to be much more complicated than in the one dimensional cases. 
Indeed, in the seminal paper \cite{KS} the authors discovered examples of surfaces with ample canonical class and semi-log-canonical singularities, which are the natural singularities to be considered for the compactification,  which are \emph{not} asymptotically GIT stable. The central point for this phenomenon is that there are semi-log-canonical singularities which have ``too big'' multiplicity compared to the one required to be asymptotically Chow stable 
(\cite{Mum2}). Nevertheless,  proper separated moduli of canonical models of surface of general type  have been recently constructed using birational geometric techniques instead of classical GIT. These compactifications are sometimes known as 
\textit{Koll\'ar-Shepherd-Barron-Alexeev (KSBA)  type moduli}. 
It is then natural to ask what is the ``differential geometric'' interpretation of these kind of moduli spaces. 

In order to discuss this last point, we first recall that GIT theory became again a main theme for the following reason: the existence of a K\"ahler-Einstein, or more generally constant scalar curvature, metric on a polarized algebraic variety is found to be deeply linked to some GIT stability notions, e.g., asymptotic Chow, Hilbert stability and in particular to the formally GIT-like notion of  ``K-stability"  introduced  in \cite{Tian2}, \cite{Do1}. Similarly to the previous discussion, asymptotic Chow stability seems to not fully capture the existence of  a K\"ahler-Einstein metric since there are examples of K\"ahler-Einstein varieties which are asymptotically Chow \textit{un}stable (\cite{KS}, \cite{Od2}). 
On the other hand, for $\Q$-Fano varieties it is indeed proved that the existence of a K\"ahler-Einstein metric implies K-polystability \cite{Berman}. 

It turns out that the notion of K-stability is also closely related with the singularities allowed in the KSBA compactifications  (\cite{Od1}, \cite{Od2}):  for  varieties  with ample canonical class, the notion of K-stability \emph{coincides} with the semi-log-canonicity property, and for Fano varieties K-(semi)stability \emph{implies} log-terminalicity. This last condition on the singularities in the Fano case it is also important for differential geometric reasons. As recently shown in \cite{DS}, it is known that Gromov-Hausdorff limits of smooth K\"ahler-Einstein Fano manifolds (and more generally of polarized K\"ahler manifolds with control on the  Ricci tensor, the injectivity radius  and with  bounded diameter) are indeed $\Q$-Fano varieties, i.e., they have at worst log-terminal singularities, and moreover they must be K-polystable,  by \cite{Berman}. 

Summing up, a central motivation of the present work was to investigate how K\"ahler-Einstein metrics and the compact moduli varieties are indeed related. Thus, motivated by our results on Del Pezzo surfaces and by the above discussion, we shall now try to state  a conjecture on moduli of K\"ahler-Einstein/K-polystable Fano varieties. 

Denote the category of algebraic schemes over $\C$ by Sch{$_\C$}, 
and let $$\mathcal{F}_h\colon \mbox{Sch}_{\C}^{o} \rightarrow Set$$ be the contravariant moduli functor which sends an object $S\in Ob( \mbox{Sch}_{\C})$ to  isomorphic classes of $\Q$-Gorenstein flat families $\X\rightarrow S$ of K-semistable $\Q$-Fano varieties with Hilbert polynomial equal to $h$ and  sends a morphism to  pull-back of families making the corresponding squared diagram commuting. Moreover, adding isomorphism (or isotropy) structure 
on this functor, we should naturally get a stack $\mathcal{M}_h$ on which we conjecture, refining \cite[Conjecture 1.3.1]{Spotti} and  \cite[Conjecture 5.2]{Od0} in the $\Q$-Fano case, the following:

\begin{conj} 
$\mathcal{M}_h$ is a KE moduli stack (cf. Definition \ref{KE moduli stack}) 
which has a categorical moduli algebraic space $$\mathcal{M}_h \rightarrow M_h, $$ where 
$M_h$ is a \emph{projective variety} (in general may not be irreducible) endowed with an ample CM line bundle. 
Especially, $M_h$ is a KE moduli variety in the sense of Definition \ref{KE moduli stack}. 

 Let $M_h^{GH}$ be the Gromov-Hausdorff compactification of the moduli space of smooth K\"ahler-Einstein Fano manifolds with Hilbert polynomial $h$. Then there is a  natural homeomorphism $$\Phi\colon M_h^{GH} \rightarrow M_h,$$ where we use the analytic topology on $M_h$. 
\end{conj}

This paper explicitly 
settles the above conjecture for ($\mathbb{Q}$-Gorenstein smoothable) log del Pezzo surface case, except the issue in the previous subsection and  the statement about the  CM line bundle. A remark  is that the 
CM line bundle \cite{PT} can be naturally regarded as a line bundle on $\mathcal{M}_h$ 
and so by ``CM line bundle on $M_h$" we mean a $\mathbb{Q}$-line bundle descended  from $\mathcal{M}_h$. 
The descent is possible for each $U_i\twoheadrightarrow U_i//G$ in the context of Definition \ref{KE moduli stack} 
since for each K-semistable $x\in U_i$ the action of the identity component of the isotropy group of $G$ on the CM line over $x$ is trivial 
by the weight interpretation of vanishing of Futaki invariant \cite{PT}. 
They canonically patch together due to the canonical uniqueness of 
the descended line bundle on each $U_i//G$. 

From the point of view of the authors, one way towards establishing the above conjecture in higher dimensions is by combining the algebraic and differential geometric techniques, as we did in this article. In many concrete situations one can hope to construct the above KE moduli stack by glueing together quotient stacks from different GIT. This also fits into the general conjecture on Artin stack 
\cite[Conjecture 1]{Alp}. 

Finally we remark that the points in the boundary $M_h \setminus M_h^{0}$ should  correspond to $\Q$-Fano varieties, admitting weak K\"ahler-Einstein metrics in the  sense of pluripotential theory \cite{EGZ}. 
This is known for $M_{h}^{GH}\setminus M_h^0$, see \cite{DS}.

\end{document}